\documentclass[12pt,amstex,leqno]{amsart}

\usepackage{latexsym}
\usepackage{amsthm}
\usepackage{amsmath}

\usepackage{amsfonts}
\usepackage{amssymb}
\usepackage{appendix}
\usepackage[mathcal]{euscript}

\newcommand\op{\operatorname{op}}

\newcommand\oti{\otimes}

\newcommand\Alg{\operatorname{\bf Alg}}
\newcommand\Set{\operatorname{\bf Set}}
\newcommand\Pos{\operatorname{\bf Pos}}

\newcommand{\id}{\operatorname{id}}

\newcommand\colim{\operatorname{\it colim}}

\newcommand{\Str}{\operatorname{\bf Str}}
\newcommand{\Met}{\operatorname{\bf Met}}

\newcommand{\Mon}{\operatorname{\bf Mon}}
\newcommand\ck{\mathcal {K}}
\newcommand\ca{\mathcal {A}}
\newcommand\cb{\mathcal {B}}

\newcommand\cd{\mathcal {D}} 

\newcommand\cl{\mathcal {L}}
\newcommand\ce{\mathcal {E}}
\newcommand\cv{\mathcal {V}}

\newcommand\cm{\mathcal {M}}

\newcommand\N{\mathbb{N}}

\newcommand\R{\mathbb{R}} 

\newcommand\var{\varepsilon}

\input xy
\xyoption{all}

\swapnumbers

\newtheorem{theorem}{Theorem}[section]
\newtheorem{lemma}[theorem]{Lemma}
\newtheorem{birk}[theorem]{Birkhoff Variety Theorem}

\newtheorem{prop}[theorem]{Proposition}
\newtheorem{corollary}[theorem]{Corollary}

\theoremstyle{definition}
\newtheorem{defi}[theorem]{Definition}

\newtheorem{assump}[theorem]{Assumption}
\newtheorem{example}[theorem]{Example}
\newtheorem{exam}[theorem]{Example}
\newtheorem{exs}[theorem]{Examples}
\newtheorem{remark}[theorem]{Remark}
\newtheorem*{rem}{Remark}

\newtheorem{nota}[theorem]{Notation}

\newtheorem{odst}[theorem]{}

\numberwithin{equation}{section}

\begin{document}

\title{Which Categories are Varieties of quantitative algebras?}
\author[Ji{\v{r}}{\'{\i}} Ad{\'{a}}mek]{Ji{\v{r}}{\'{\i}} Ad{\'{a}}mek$^{\ast)}$}
\dedicatory{\rm Czech Technical University, Prag\\
Technical University of Braunschweig}

\begin{abstract}
Classical varieties were characterized by Lawvere as the categories with effective congruences and a varietal generator:
an abstractly finite regular generator which is regularly projective (its hom-functor preserves regular epimorphisms).  We
characterize varieties of quantitative algebras  of Mardare, Panangaden and Plotkin analogously as metric-enriched  categories.  We introduce the concept of a subcongruence  (a metric-enriched analogue of a congruence) and the corresponding subregular epimorphisms obtained via colimits of subcongruences. Varieties of quantitative algebras are precisely the metric-enriched categories with effective subcongruences and a subvarietal generator: an abstractly finite  subregular generator which is subregularly projective (its hom-functor preserves subregular epimorphisms).
\end{abstract}

\maketitle
\footnote{$^{\ast)}$ Supported by the Grant Agency of Czech Republic under the grant 22-02964S.}

\section{Introduction}\label{sec1}

For the semantics of probabilistic programs Mardare et al. introduced in \cite{MPP17} varieties (aka 1-basic varieties) of quantitative algebras, and their properties have been studied in a number  of articles, e.g.\ \cite{A}, \cite{B1}, \cite{B2} and \cite{B3}. 
The aim of our paper is to characterize those varieties as categories enriched over the category $\Met$ (of extended metric spaces and nonexpansive maps) which have effective subcongruences and a `well-behaved' generator.
This is completely analogous to Lawvere's characterization of (classical) varieties as exact categories having a varietal generator (\cite{L}). A similar characterization of varieties of ordered algebras has been presented in \cite{A1}. These two results are recalled in Section 2.  We conclude that the choice Mardare et al. made when defining
quantitative algebras was the 'right' one: they require the $n$-ary operations
on a space $A$ to be nonexpanding with respect to the categorical product $A^n$  (having the maximum metric, see 3.2). However, $\Met$ is not
cartesian closed:
this is a symmetric monoidal closed category with the tensor product
given by the addition metric (see 3.1). Thus, for example, the monoid $(\mathbb{R}, +,0)$ with the Eucledian metric
(which is a monoid in $\Met$) is \emph{not} a quantitative monoid.  (Addition fails to be nonexpanding with respect to the maximum metric: consider the elements $(0,1)$ and $(1,2)$ of distance $1$ in $\R^2$, whereas their additions have distance $2$.)
The fact that varieties of quantitative algebras have, as we demonstrate below, a characterization completely  analogous to the classical varieties (and the ordered varieties) seems to justify their definition. 

The main 
concept in metric-enriched categories we introduce is that of subcongruence, a natural analogy of congruences in ordinary categories. A typical congruence on an object X of an ordinary category is the kernel pair of a morphism $f \colon X \to Y$; effectivity of congruences 
means that there are no other examples of congruences than kernel pairs. In metric-enriched categories we
work with the $\var$-kernel pair of a morphism $f$ for every real number $\var \geq 0$:  this is the parallel pair
$l, r \colon P \to X$ universal with respect to the property that the distance of $f \cdot l, f \cdot r \colon P \to Y$ is at most $\var$.
The \emph{kernel diagram} of $f$ is
a weighted diagram collecting all of its $\var$-kernel pairs. A typical subconguence is the kernel diagram of a morphism. Effectivity of subcongruences means that there are no other  examples. This is the topic of Section 3. 

A colimit of a subcongruence is determined by a certain morphism. In analogy to regular epimorphisms, we call
a morphism $f$ a \emph{subregular epimorphim} if there is a subcongruence whose colimit is determined by $f$.
(In varieties of quantitative algebras these are precisely the surjective homomorphisms.) An object whose hom-functor
preserves subregular epimorphisms is called a \emph{subregular projective}. 
We characterize varieties of quantitative algebras as the metric-enriched categories with effective subcongruences and a subvarietal generator: an abstractly finite subregular generator which is subregularly projective (Theorem \ref{T:main}).

 As an application we show that monoids in the monoidal category $\Met$ do not form a category equivalent to a variety of quantitative algebras (Example \ref{E:count}).

\vspace{4mm}

\textbf{Related work}

In the fundamental paper of Bourke and Garner \cite{BG} congruences in general enriched categories are presented.
The basic weight of Definition \ref{D:basic} leads to an instance of the category $\mathcal F$ in loc. cit. by adding to
our category $\mathcal B$ a morphism with domain $a$ and a new codomain. By considering $\Met$ as the base category,
$\mathcal F$ as above, and the category $\ck$ in op.cit. as our $\cb$, we get the following
 translation of the concepts of Bourke and Garner: $\mathcal F$-monic morphisms are our isometric embeddings (Definition \ref{D:iso}), and $\mathcal F$-quotient maps are our subregular epimorphisms (Definition \ref{D:subr}), etc.
 However, our concept of subcongruence appears to be simpler than the general $\mathcal F$-congruence of op. cit.
 The characterization of varieties of quantitative algebras we present in Theorem \ref{T:main} does not seem to follow from
 the general results of Bourke and Garner.	
\vspace{4mm}

\textbf{Ackowledgement.} The author is grateful to Henning Urbat for his very helpful suggestions.

\section{Varieties of Classical or Ordered Algebras}\label{sec2}

The aim of this section is to shortly recall the characterization of varieties and ordered varieties, so that the reader can discern the analogy to the concepts of subregular epimorphism and subcongruence introduced in the next section. It is possible to skip this section without losing the context.

\begin{assump} Throughout the paper $\Sigma$ denotes a finitary signature: a collection of sets $\Sigma_n$ of $n$-ary operation symbols for $n \in \Bbb{N}$.
\end{assump}

By a (classical) \emph{variety} of algebras is meant a class of $\Sigma$-algebras  presented by a set of equations between terms. Such as monoids, lattices, Boolean algebras etc. Lawvere introduced the following concept
(weaker than that of a finitely generated object):

\begin{defi}[\cite{L}]\label{D:abs}
	An object $X$ with copowers is \emph{abstractly finite} if every morphism from $X$ to a copower $\coprod\limits_{M} X$ factorizes through a finite subcopower $\coprod\limits_{M'} X$ (where $M'$ is a finite subset of $M$). 
\end{defi}

\begin{exam}
	In $\Set$ these are precisely the finite sets. In the category $\Pos$ of posets these are precisely the posets with a finite number of connected components (thus,
	the above concept if much broader than that of a finite poset).
\end{exam}

We recall the concept of a congruence. One does not need to assume finite completeness of the category $\ck$: a congruence is a collectively monic parallel pair $l$, $r\colon R\to X$ that every hom-functor $\ck(S,-)$ takes to an equivalence relation on the set $\ck (S,X)$.
More detailed:

\begin{defi}\label{R:cong}
	(1) A  \emph{relation}
		on an object $X$ of an  ordinary category is a collectively monic (ordered) pair of morphisms $l$, $r\colon R\to X$. (If the product $X \times X$ exists, then $\langle l, r\rangle \colon R \to X\times X$ represents a subobject of $X \times X$.)
	
	(2)	A \emph{congruence} on $X$  is  a relation $l$, $r\colon R\to X$ which is:
		
		(a) Reflexive. This can be expressed by saying that for every morphism $s \colon S \to X$ the pair
		$(s, s)$ factorizes through $(l, r)$. Equivalently: $l, r$ are split epimorphisms with a common splitting.
		
		(b) Symmetric. The pair $(r, l)$ factorizes through $(l, r)$.
		
		(c) Transitive. Given morphisms $u_0, u_1, u_2 \colon S \to X$ such that both $(u_0,u_1)$ and $(u_1,u_2)$
		factorize through the pair $(l,r)$, then also $(u_0,u_2)$ factorizes through that pair. 
		
		This last condition
		simplifies in the presence of pullbacks as follows:
		given the pullback $P$ of $r$ and $l$ (on the left) the pair $(	l\cdot \bar l$, $r\cdot \bar r)$ factorizes through $(l$, $r)$:
		$$
		\xymatrix@C=0.5pc@R=1.1pc{
			&&P\ar	[dl]_{\bar l}\ar@{}[d]|{\vee} \ar [dr]^{\bar r}&&\\
			& R\ar[dl]_{l} \ar[dr]_r&& R \ar [dl]^l \ar[dr]^r&\\
			X &&X&&X
		}
\hspace {5mm}	
		\quad
		\xymatrix@C=0.5pc@R=1.1pc{
			&&P\ar	[dl]_{\bar l}\ar[dd]_p \ar [dr]^{\bar r}&&\\
			& R\ar[dl]_{l}&& R  \ar[dr]^r&\\
			X && R\ar[ll]^{l} \ar[rr]_r&&X
		}
		$$
	\end{defi}	
	
	\begin{rem} A typical congruence is the kernel pair of a morphism $f \colon X \to Y$, which is the universal
	relation with  respect to $l \cdot f = r \cdot f$. That is, a pullback of $f$ with itself.

	 A category has \textit{effective congruences} if every congruence is a kernel pair of some morphism.
		\end{rem}
 
	\begin{remark}\label{R:gen}
	 An  object $G$ with copowers is
	
	(1)  A \emph{regular generator} if for every object $X$ the canonical  morphism $$[f] \colon \coprod\limits_{f\colon G\to X} G\to X$$ is a regular epimorphism.
		
	(2) A \emph{regular projective} if its hom-functor preserves regular epimorphisms: every morphism from $G$ to $X$
	factorizes through each regular epimorphism $Y \to X$.
	
\end{remark}

The following result is a minor sharpening of the characterization of varieties due to Lawvere \cite{L}.

\begin{theorem}[{\cite[Thm. 32]{A1}}]\label{T:2.6}
	A category is equivalent to a variety iff it has 
	
	(1) Reflexive coequalizers.
	
	(2) Effective congruences.
	
	(3) A \emph{varietal generator}: an abstractly  finite, regularly projective regular generator with copowers.
\end{theorem}

For order-enriched categories (enriched over  the cartesian closed category $\Pos$ of posets), an analogous result has recently been proved (\cite{AR2}). We work with $\Sigma$-algebras acting on a poset so that the operations are all monotone, and morphisms are the monotone homomorphisms. Varieties are presented by inequations between terms.
These categories are order-enriched: each hom-set is ordered pointwise. 

An object $G$ in an order-enriched category \emph{has tensors} if its hom-functor into $\Pos$ has a left adjoint (asigning to each poset $P$ an object $P \otimes G$). For every object $X$ we have the \emph{evaluation morphism} $\widehat{\id} \colon \ck (G,X) \otimes G \to X$,
which is the adjoint transpose of $\id \colon  \ck (G,X) \to \ck(G,X)$.

Given an order-enriched category, by a \emph{relation} on an object $X$ is meant a pair $l,r \colon R \to X$ of morphisms which is a collective order-embedding: given $u_1, u_2 \colon U\to R$ with $l\cdot u_1 \leq l\cdot u_2$ and $r\cdot u_1 \leq r\cdot u_2$, it follows that $u_1 \leq u_2$.

A \emph{coinserter} of a parallel pair $u_0$, $u_1 \colon U \to V$ is the universal morphism $c\colon V\to C$ with respect to $c\cdot u_0\leq c\cdot u_1$.

 \begin{defi}
A \emph{subregular epimorphism} is a morphism which is the coinserter of some reflexive parallel pair.
\end{defi}

In varieties of ordered algebras these are precisely the surjective morphisms.

\begin{defi}[ \cite{AR2}]
	Let $\ck$ be an order-enriched category.	 An  object $G$ with tensors is
	
	(1)  A \emph{subregular generator} if for every object $X$ the evaluation  morphism $\widehat{\id} \colon \ck (G,X) \otimes G \to X$  is a subregular epimorphism.
	
	(2) A \emph{subregular projective} if its hom-functor preserves regular epimorphisms: every morphism from $G$ to $X$
	factorizes through each subregular epimorphism $Y \to X$.
	
	(3) A \emph{subeffective projective} if its hom-functor preserves coinserters of reflexive pairs.
\end{defi}

\begin{theorem} (\cite{AR2}, Thm. 4.8) Varieties of ordered algebras are, up to equivalence, precisely the order-enriched categories with reflexive coinsetrters and an abstractly finite subregular generator which is subeffectively projective.
\end{theorem}

A certain disadvantage of this result is that the concept of a subeffective projective is not very intuitive. In contrast, subregular projectives are completely analogous to regular projectives in ordinary varieties: they are the objects $X$
in varieties of ordered algebras such that for every surjective morphism $e \colon A \to B$ all morphims from $X$ to $B$ factorize through $e$. This has been improved in \cite{A1} by introducing effectivity of subcongruences which we now recall.

The role of the kernel pairs of a morphism $f\colon X\to Y$ is played in $\Pos$-enriched categories by \emph{the subkernel pair}: the universal pair $l,r \colon R\to X$ with respect to  $f\cdot l\leq f\cdot r$. Every subkernel pair is a relation which is not only reflexive, but has the following stronger property we called order-reflexivity:

\begin{defi}[\cite{A1}]\label{D:subcon}
	Let $\ck$ be an order-enriched category.
	A relation $l,r \colon R\to X$ is 
	
	(1 )\emph{Order-reflexive} if every pair $s_1\leq s_2 \colon S\to X$ factorizes through it.
	
	(2) A \emph{subcongruence} is order-reflexive and transitive .
	
	The category has \emph{effective subcongruences} if every subcongrience is the subkernel of some morphism.
\end{defi}

\begin{theorem}[{\cite[Thm. 5.8]{A1}}]\label{T:pos} 
	An order-enriched category is equivalent to a variety of ordered algebras iff it has 
	
	(1) Reflexive coequlizers and subkernel pairs.
	
	(2) Effective subcongruences.
	
	(3) A \emph{subvarietal generator}: 
	an abstractly finite, subregularly projective, subregular generator with tensors.
\end{theorem}

A related result is presented by Kurz and Velebil \cite{KV}: they characterize varieties of ordered algebras as the monadic  categories over $\Pos$ for strongly finitary monads.

\section{Subregular Epimorphisms and Subcongruences}\label{sec3}

The role that regular epimorphisms play in varieties is played in the metric-enriched case by subregular epimorphisms which  we introduce now. For example, in a variety of quantitative algebras subregular epimorphisms are precisely the surjective homomorphisms (Corollary \ref{C:var-epi}). The role of congruences is played by weighted diagrams we call subcongruences (Definition \ref{D:proeq}).

\begin{assump} 
Throughout the rest of our paper we work with categories and functors enriched over $\Met$.
This is the closed monoidal category with

Objects: (extended) metric spaces, defined as usual except that the distance $\infty$ is allowed.

Morphisms: nonexpanding functions $f\colon X\to Y$, i.e. for all $x, x' \in X$ we have $d(x,x')\geq d\big(f(x), f(x')\big)$.

Unit: a singleton space $I$.

Tensor product $X \otimes Y$: cartesian product with the \emph{addition metric}:
$$d((x,y), (x', y'))=d(x,x') + d(y,y').$$
\end{assump}


Every set $X$ is considered to be the \emph{discrete space}: all distances  are $\infty$ or $0$. The underlying set of a space $M$ is denoted by $|M|$.

\begin{remark} 
	(1) A product $X \times Y$ in $\Met$ is the cartesian product with the \emph{maximum metric}:
	$$d((x,y), (x', y'))= \max \{d(x,x'), d(y,y')\}.$$

(2) A \emph{$($metric$)$-enriched category} is a category with a metric on each hom-set such that composition is nonexpanding (with respect to the addition metric). 
A \emph{$($metric$)$-enriched functor} $F\colon \ck \to \cl$ is a functor which is locally nonexpanding: for all parallel pairs $f$, $g\colon X\to Y$ in $\ck$ we have $d(f,g)\geq d(Ff, Fg)$. The concept of a metric-enriched natural transformation coincides with natural transformation between the underlying ordinary functors.

The category of all enriched functors from $\ck$ to $\cl$ is denoted by
$$
[\ck, \cl]\,.
$$
It is enriched by defining the distance of natural transformations $\sigma, \tau \colon F\to F'$ to be $\sup\limits_{K\in \ck} d(\sigma_K, \tau_K)$.
\end{remark}

\begin{exam} \label{E:alg}
(1) Mardare et al. \cite{MPP17} introduced the category
$$
\Sigma\text{-}\Met
$$
of \emph{quantitative algebras}. Objects are $\Sigma$-algebras $A$ acting  on a metric space such that the operations
$$
\sigma_A \colon A^n \to A \qquad (\sigma \in \Sigma_n)
$$
are nonexpanding with respect to the maximum metric.
Morphisms are the nonexpanding homomorphisms.

This is an enriched category: the distances of parallel morphisms are given by the supremum metric:
$$
d(f,g) =\sup_{a\in A} d \big( f(a), g(a) \big) \quad \mbox{for morphisms} \quad f, g \colon A\to B\,.
$$

(2) The forgetful functor  $U_\Sigma \colon \Sigma$-$\Met \to \Met$ is an enriched functor.
\end{exam}

It may come as a surprise that Mardare et al. do not use the addition  metric for quantitative algebras. However, the fact that varieties have, as we prove below, a characterization completely  analogous to what we have seen in Section \ref{sec2} justifies the definition.

Recall from Section \ref{sec2} that an object $G$ is abstractly finite if every morphism $f\colon G\to \coprod\limits_{M} G$ factorizes through a finite subcopower $\coprod\limits_{M'} G$
 ($M' \subseteq M$ finite). This is weaker than the concept of a finitely generated object (whose hom-functor preserves directed colimits of subobjects). And  in the setting of metric-enriched categories abstract finiteness fits well:
 
 \begin{example}\label{E:abst} 
 (1) In the category $\Met$ the only finitely generated object is the empty space (\cite{AR3}).
 In contrast, every finite space is abstractly finite.
 
 More generally, every metric space $G$ with finitely many connected components is abstractly  finite. Indeed, let $G = G_1 + \dots + G_k$ with $G_i$ connected. For every morphism $f\colon G \to \coprod\limits_{M} G$ and each $i= 1, \dots , k$ the image $f[G_i]$ is connected, thus, it lies in one of  the summands, say the summand corresponding to $m_i \in M$.
 Then $f$ factorizes through $\coprod\limits_{M'} G$ for $M'=\{m_1, \dots, m_k\}$.
 
 (2) The free algebra on one generator in $\Sigma$-$\Met$ is an abstractly finite object (Example \ref{E:abs}).
 \end{example}
 
Let us recall the concept  of a \emph{weighted colimit} in metric-enriched categories $\ck$, due to Borceux and Kelly \cite{BK}. We are given a diagram $D\colon \cd \to \ck$ and a weight $W \colon \cd^{\op} \to \Met$, both enriched functors. The weighted colimit is an object of $\ck$
$$
C=\colim\limits_W D 
$$
together with isomorphisms (in $\Met$)
$$
\psi_X \colon \ck(C, X) \to \big[ \cd^{\op}, \ck\big] \big(W, \ck (D-, X)\big)
$$ 
natural in $X \in \ck$.

The \emph{unit} of $\colim\limits_W D$ is the natural transformation
$$
u = \psi_C (\id_C)\colon W\to \ck (D -,C)\,.
$$

Ordinary colimits of $D\colon \cd\to \ck$ are considered to be conical, i.e., weighted by the trivial weight (constant with value $I$).

Dually, given a diagram $D\colon \cd \to \ck$ and a weight
$W\colon \cd \to \Met$, the \emph{weighted limit} is an object $L=\lim\limits_W D$ together with natural isomorphisms
$$
\psi_X \colon \ck (X,L) \to [\cd, \ck] \big(W, \ck(X, D-)\big)\,.
$$

\begin{remark}\label{R:3.4}
All weighted colimits and weighted limits of diagrams in $\Met$ exist (\cite[Example 4.5]{AR3}).
\end{remark}

\begin{exam}\label{E:ten} 
\emph{Tensor product.} Given a metric space $M$ and an object $K$ of $\ck$, the object
$$
M\otimes K
$$
is the weighted colimit of the diagram $D\colon 1\to \ck$ representing $K$, weighted by $W$: $1^{\op}\to \Met$  representing $M$.
In other words,  there is a bijection
$$
\frac{M\otimes K\to X}{M\to \ck(K,X)}
$$
natural in $X\in \ck$. The unit $u\colon W\to \ck (D-, M\otimes K)$ is given by the morphism $M\to \ck(K, M\otimes K)$ adjoint to $\id_{M\otimes K}$.

An object $K$ \emph{has tensors} if $M\otimes K$ exists for all spaces.
 Equivalently: the hom-functor $\ck(K, -)\colon \ck \to \Met$ has an (enriched) left adjoint $-\oti K\colon \Met \to \ck$.
 
 If $M$ is a set (a discrete space), then
 $$
 M\otimes K= \coprod_M K\,.
 $$
\end{exam}

\begin{nota}\label{N:hat} 
(1)The morphism of $\ck$ corresponding to $f\colon M\to \ck (K,X)$ (in $\Met$) is denoted by
$$
\hat f\colon M\otimes K\to X\,.
$$
\end{nota}

\begin{example}\label{E:n}
(1) The morphism $\id \colon \ck (G,X) \to \ck (G,X)$ yields
the \emph{evaluation morphism}
$$
\hat \id\colon \ck(G,X)\otimes G\to X\,.
$$
For example if $\ck =\Met$, then $\widehat\id$ assigns to $(f,x)$ the value $f(x)$.

(2) Every morphism $f\colon X\to Y$ yields a map  $f_0 \colon I\to \ck(X,Y)$ in $\Met$ with
$$
f= \widehat f_0\,.
$$
\end{example}

 \begin{defi}\label{D:iso} 
(1) A morphism $p\colon P\to X$ is an \emph{isometric embedding} if for all parallel pairs $u_1$, $u_2\colon U\to P$ we have 
$$
d(u_1, u_2) = d(p\cdot u_1, p\cdot u_2)\,.
$$

(2) Dually, a morphism $p \colon P \to X$ is a \emph{quotient} if for all parallel pairs $u_1$, $u_2\colon X\to U$ we have
$$
d(u_1, u_2) = d(u_1\cdot p, u_2\cdot p)\,.
$$
For example, every surjective morphism in $\Met$ is a  quotient.

(3) A parallel pair $p$, $p'\colon P\to X$ is \emph{isometric}
if for all parallel pairs $u_1, u_2\colon U \to P$ we have 
$$
d(u_1, u_2)=\max \{d(p\cdot u_1, p\cdot u_2), d(p'\cdot u_1, p'\cdot u_2)\}\,.
$$
In case conical products exist, this is equivalent to statimg that
  the morphism $\langle p, p'\rangle \colon P\to X\times X$ is an isometric embedding.

  A \emph{relation} on an object $X$ is an isometric pair $p, p' \colon P \to X$.

  \end{defi}
  
  \begin{nota}
 We denote by $\var$, $\var'$ etc. non-negative real  numbers.
  \end{nota}
  \begin{example} \label{E:ker}
  An important example of a weighted limit is the \emph{$\var$-kernel pair} of a morphism $f\colon X\to Y$. This is a universal pair $p_1$, $p_2\colon P\to X$  with respect to $d(f\cdot p_1, f\cdot p_2)\leq \var$. That is, an isometric pair such that  given a pair  $q$, $q'\colon Q\to X$ with $d(f\cdot q, f\cdot q')\leq \var$, then this pair factorizes through $(p_1, p_2)$.

  This is the weighted limit of the following diagram
  
  $
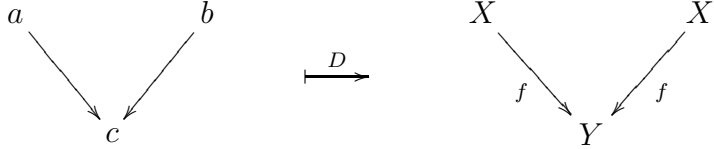
\xymatrix@C=2pc@R=1pc{
a \ar[ddr]& & b \ar[ddl]&\\
&&& \ar@{|->}[r]^D&\\
&c&&&
}$
\qquad
$
\xymatrix@C=2pc@R=1pc{
X \ar[ddr]_f && X \ar[ddl]^f\\
&&\\
& Y &
}
$

\noindent
weighted  by the following weight
$$
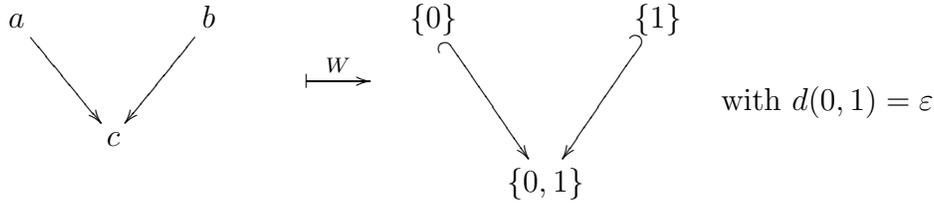
\xymatrix@C=2pc@R=1pc{
a \ar[ddr]& & b \ar[ddl]&\\
&&& \ar@{|->}[r]^W&\\
&c&&&
}
\ \,
\xymatrix@=1pc{
\{0\} \ar@{_{(}->}[ddr] && \{1\} \ar@{_{(}->}[ddl]&\\
&&&\!\!\mbox{with}\ d(0,1) =\varepsilon\\
& \{0,1\} &&&
}
$$
\end{example}

An enriched category \textit{has $\var$-kernel pairs} if for every morphism and every $\var$ an $\var$-kernel pair exists.

In $\ck=\Met$ the $\var$-kernel pair of $f\colon X\to Y$ is given by the subspace of $X\times X$ on all pairs $(x_1, x_2)$ with $d\big(f(x_1), f(x_2)\big) \leq \var$.

\begin{remark}\label{R:ker}  
  The $\var$-kernel pair $(p_1, p_2)$ of $f\colon X\to Y$ is clearly reflexive and symmetric, but usually not transitive. 
  
  It has also the following (somewhat unexpected) universal property: every parallel pair $q_1$, $q_2\colon Q\to Y$ of distance $\var$ factorizes through $(p_1, p_2)$.
  
Indeed, since composition is nonexpanding, we have that
$$
 d(q_1, q_2)=\var \quad \mbox{implies} \quad  d(f\cdot q_1, f\cdot q_2)\leq \var\,,
$$
thus there exists $u\colon Q\to P$ with $q_i = p_i\cdot u$.
We call this property $\var$-reflexivity below because it generalizes reflexivity (Definition \ref{R:cong}). 
\end{remark}

\begin{defi}\label{D:eps}
 A parallel pair $p_1, p_2\colon P\to X$ is \emph{$\var$-reflexive} if every pair $q_1, q_2 \colon Q\to X$ with $d(q_1, q_2) \leq \var$ factorizes through it.
 \end{defi}

We now introduce the kernel diagram of a morphism $f$: it consists of $\var$-kernel pairs of $f$ for all $\var$. To make this precise, we work with a weight that is so `basic' in our development that we call it  the basic weight, denoted by  $B\colon \cb^{\op}\to \Met$.

\begin{defi} \label{D:basic}
Let $\cb$ be the discretely enriched category obtained from the chain  of all 
 real numbers $\var\geq 0$ by adding two parallel cocones 
with codomain $a$ to it:
$
l_\var, r_\var \colon \var \to a.
$
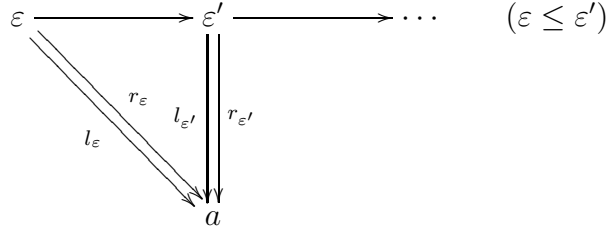
$$
\xymatrix@=5pc{
\varepsilon\ar[r] 
\ar@<0.4ex>[dr]^{r_\varepsilon}
\ar@<-0.4ex>[dr]_{l_\varepsilon}
& 
\varepsilon' \ar[r]
\ar@<-0.4ex>[d]_{l_{\varepsilon'}}
\ar@<0.4ex>[d]^{r_{\varepsilon'}}
 & \cdots \qquad (\varepsilon \leq \varepsilon')\\
 & a &&
}
$$

The \emph{basic weight} is the functor
$$
B\colon \cb^{\op} \to \Met
$$
defined on objects by $ Ba =\{\ast\}$ {and} 
$$ B\var =\{0,1\} \quad \mbox{with}\quad d(0,1)=\var\,.
$$
It is defined on the morphisms $\var \to \var'$ as $\id_{\{0,1\}}$, and on the morphisms $l_\var$, $r_\var$ as follows:
$$
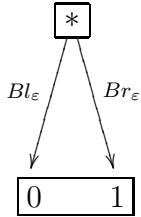
\xymatrix@C=0.51pc@R=4pc{
& *+[F]{\ast} 
\ar [dl]_{Bl_{\varepsilon}}
\ar [dr]^{Br_\varepsilon}
&&\\
&&}
$$
\vskip-2.9mm\hskip 5.1cm $\boxed {0\qquad\,1}$
\end{defi}

\begin{remark}\label{R:3.12}
A colimit of a diagram $D\colon \cb\to \ck$ weighted by $B$ is called a \emph{basic colimit}. Given such a colimit
$$
C=\colim_B D
$$
we have the unit
$$ u\colon W\to \ck (D-, C)\,.
$$
Its component
$$
u_a \colon \{\ast\} \to \ck (Da, C)
$$
determines a morphism $f\colon Da \to C$.
\end{remark}

\begin{prop}\label{P:COM}
To give  a basic colimit of $D\colon \cb \to\ck$ means  to give an object $C$ and  a morphism $f\colon Da \to C$ universal with respect to the following \emph{compatibility condition} 
\begin{equation}
d(f\cdot Dl_\var, f\cdot Dr_\var)\leq \var \quad \mbox{for each $\var$}\,.\tag{COM}
\end{equation}
\end{prop}

\begin{rem}
`Universal' means that 
$f$ is a quotient (Definition \ref{D:iso}) such that 
every morphism $f'\colon Da \to C'$ satisfying \thetag{COM} factorizes through $f$.
\end{rem}

\begin{proof}
Let $u$ be the unit as in Remark \ref{R:3.12}.
As above, put $f= u_a(\ast) \colon Da\to C$. The components $u_\var$ are, due to the naturality of $u$, given by
$$
u_\var (0) = f\cdot D l_\var \quad \mbox{and}\quad u_\var (1) = f\cdot Dr_\var\,.
$$
Thus $u_\var$ is determined by $f$, and since $u_\var$ is nonexpanding and $d(0, 1) =\var$ in $B\var$, we see that $f$ satisfies the compatibility  condition.

Conversely, whenever a morphism  $f\colon Da \to C$ satisfies \thetag{COM}, there is a unique natural transformation
$$
u\colon B\to \ck (D-, C)\quad \mbox{with}\quad u_a(\ast) =f\,.
$$
Indeed, the components $u_\var$ are determined by the above equations.

In case $f$ is universal with respect to \thetag{COM}, we obtain a natural isomorphism
$$
\psi \colon \ck (C,X) \to \big[\cd^{\op}, \ck\big] \big( B, \ck(D-, X)\big)
$$
assigning to $h\colon C\to X$ the unique natural transformation whose $a$-component is given by $\ast\mapsto h\cdot f$.
\end{proof}

\begin{defi}
The universal morphism $f\colon Da\to C$ of the above proposition is called the \emph{colimit map} of the weighted colimit $\colim_B D$.
\end{defi}

The following example stems from \cite{ADV} where the basic weight has been introduced:

\begin{example}\label{E:bas}
Every metric space $M$ is a basic colimit  of discrete spaces. Define a diagram
\begin{align*}
D_M &\colon \cb \to \Met
\\
\intertext{by}
D_M a&= \vert M\vert\,, \quad \mbox{the underlying set of \ \, $M$},
\\
\intertext{and}
D_M\var &\subseteq \vert M \vert ^2 \quad \mbox{with} \quad (x_1, x_2)\in D_M \var \quad \mbox{iff}\quad d(x_1, x_2)\leq \var\,.
\end{align*}
On morphisms $\var \to \var'$ we use the inclusion maps $D_M \var \hookrightarrow D_M \var'$, and we define
$$
D_M l_\var, D_M r_\var \colon D_M \var \to \vert M\vert
$$
to be the left and right projection, resp.

The colimit map of $D_M$ is $\id \colon \vert M\vert \to M$. 
\end{example}

  \begin{defi}\label{D:ker}
Let $\ck$ be an enriched category with $\var$-kernel pairs. Given a morphism $f\colon X\to Y$ its \emph{kernel diagram} is the following basic diagram $D\colon \cb \to \ck$: Put
$Da =X$, and let the object $D\var$ and the morphisms $Dl_\var$ and $Dr_\var$ be given by the $\var$-kernel pair of $f$ below:
$$
\xymatrix@1{
D\varepsilon \ar@<0.6ex>[r]^{Dr_\varepsilon}
\ar@<-0.6ex>[r]_{Dl_\varepsilon}
& \ X\ar[r]^f &\ Y
} \qquad \mbox{(for all $\varepsilon$).}
$$
Given $\var\leq \var'$, then $D(\var \to \var')$ is the unique factorization of $(Dl_\var, Dr_\var)$ through $(Dl_{\var'}, Dr_{\var'})$.
\end{defi}

\begin{example}
(1) The diagram $D_M$ of Example \ref{E:bas} is the kernel diagram of the identity morphism on $M$.

(2) For constant morphisms $f\colon X \to Y$ the kernel diagram is the constant functor with value $X \times X$.
\end{example}

\begin{defi}\label{D:subr}
A \emph{subregular epimorphism} is a morphism which is the colimit map of some basic diagram.
\end{defi}

\begin{example}
For every space $M$ the identity map from $|M|$ to $M$ is a subregular epimorphism in $\Met$: see Example \ref{E:bas}. We will see below that the subregular epimorphisms are precisely the surjective ones.
\end{example}

\begin{remark}\label{R:pro}
(1) Every subregular epimorphism $e\colon X\to Y$ is a  quotient (Definition \ref{D:iso}).
 
 In particular:
 $$
 \mbox{subregular epi}\ \Rightarrow\  \mbox{epi}.
 $$
 
(2)  In  an enriched category with coequalizers we have:
$$
 \mbox{regular epi}\ \Rightarrow \mbox{\ subregular epi}.
 $$
Indeed, if $f\colon X\to Y$ is a coequalizer of $p$, $q\colon P\to X$, then we define $D\colon \cb \to \ck$ by $D\var =P$ for all $\var$ and $Da=X$. Put $D(\var \to \var')=\id_{P}$ and $Dl_\var =p$ and $D r_\var =q$ for all $\var$. Then $f$ is the colimit  map of $ D$. 
\end{remark}

Observe that kernel diagrams are unique up-to natural isomorphism: functors $D, D'\colon \cb \to \ck$ are kernel diagrams of the same morphism iff it they are naturally isomorphic. 

\begin{lemma}\label{L:pro} 
Let $\ck$ be an enriched category with $\var$-kernel pairs.
Every subregular epimorphism is the colimit map of its kernel diagram.
\end{lemma}

\begin{proof}
Let $f\colon X\to Y$ be a subregular epimorphism, i.e.\ the colimit map of some $D\colon \cb \to \ck$ weighted by $B$. Denote by $D_f\colon \cb \to \ck$ the kernel diagram.
We know that $f$ is a  quotient.
 We thus only need to prove that every morphism $f'\colon X \to Y'$ satisfying \thetag{COM} for $D_f$:
\begin{equation}\label{e1}
d(f'\cdot D_f l_\var, f' \cdot D_f r_\var)\leq \var \qquad (\mbox{for all $\var$}).
\end{equation}
factorizes through $f$. For that it is sufficient to verify that \eqref{e1} implies (COM) for $D$:
\begin{equation}\label{e2}
d(f'\cdot Dl_\var, f' \cdot D r_\var)\leq \var \qquad (\mbox{for all $\var$}).
\end{equation}

Since $D_f l_\var$, $D_f r_\var$ is the $\var$-kernel pair of $f$,
and $d(f\cdot Dl_\var, f \cdot D r_\var)\leq \var$, there exists $h_\var \colon D\var \to D_f\var$ with $Dl_\var = D_f f_\var \cdot h_\var$ and $Dr_\var = D_f r_\var \cdot h_\var$. Thus from \eqref{e1} we get
\begin{align*}
d(f'\cdot Dl_\var, f' \cdot D r_\var) &= d(f'\cdot D_f l_\var \cdot h_\var , f'\cdot D_f r_\var \cdot h_\var)\\
&\leq d(f'\cdot D_fl_\var, f'\cdot D_f r_\var)\leq \var\,.
\qedhere
\end{align*}
\end{proof}

\begin{remark}\label{R:epi} 
We will  use the following well-known (and easily verified) fact from Universal Algebra: given a surjective homomorphism $e\colon X\to Y$ and a homomorphism $f\colon Y\to Z$ whose underlying map factorizes (in $\Set$) through $e$:
$$
\xymatrix{ 
|X| \ar[d]_f \ar[r]^e & |Y| \ar@{-->}[dl]^h\\
|Z| &
}
$$
it follows  that $h\colon Y\to Z$ is a homomorphism.
\end{remark}

\begin{prop}\label{P:sur} 
Subregular epimorphisms in  $\Met$, and more generally in $\Sigma$-$\Met$, are precisely the surjective morphisms.
\end{prop}

\begin{proof}
Since $\Met$ is a special case of $\Sigma$-$\Met$ (with $\Sigma=\emptyset$), we provide a proof for  $\Sigma$-$\Met$.

(1) Let $f\colon X\to Y$ be a subregular epimorphism. For the subspace $m\colon Y' \hookrightarrow Y$ on the image of 
$f$, let $f'\colon X\to Y'$ be the codomain restriction of $f$. Then $Y'$ is a subalgebra of $Y$: given $\sigma \in \Sigma_n$ and an $n$-tuple $y_i = f(x_i)$ in  $Y'$ we have $\sigma_Y  (y_i) = f\big(\sigma_X(x_i)\big) \in Y'$. Also, $f'$ is  a nonexpanding homomorphism (since $Y'$ is a subspace of $Y$). Moreover, if $f$ is the colimit map of $\colim\limits_B D$,
then $f'$ satisfies \thetag{COM}: from $d(f\cdot Dl_\var, f\cdot Dr_\var)\leq \var$ we clearly get $d(f'\cdot Dl_\var, f'\cdot Dr_\var) \leq \var$.
Consequently, $f'$ factorizes through $f$:
$$
\xymatrix@=4pc{ 
X \ar[d]_{f'} \ar[r]^f & Y \ar@<0.6ex>[dl]^h\\
Y'\ar@<0.6ex>@{^{(}->}[ur]^m &
}
$$
via some $h\colon Y\to Y'$ in $\Sigma$-$\Met$. Since
$m$ and $h$ are homomorphisms, from 
$$
(m\cdot h)\cdot  f = m\cdot f' =f
$$
and the fact that $f$ is a quotient we get  $m\cdot h =\id$. Thus $Y'=Y$, and  $f$ is surjective.

\vskip 1mm
(2) Let $f\colon X\to Y$ be a surjective morphism of  $\Sigma$-$\Met$. It satisfies \thetag{COM} with respect to its  kernel diagram $D\colon \cb \to \Sigma$-$\Met$. We prove that for every $f'\colon X\to Y$ satisfying that condition, too:
\begin{equation*}
d(f'\cdot Dl_\varepsilon, f'\cdot Dr_\varepsilon) \leq \varepsilon
\tag{C}
\end{equation*}
there is a factorization through $f$ in $\Sigma$-$\Met$. Then $f$ is clearly universal, hence, a subregular epimorphism.

For $\var =0$ the  condition \thetag{C} yields  the following implication:
$$
f(x_1) = f(x_2) \quad \mbox{implies}\quad  f'(x_1) = f'(x_2) \qquad (x_1, x_2\in X)\,.
$$
We thus get a well-defined map $h\colon |Y'|  \to |Y|$ by
$$
h\big( f(x)\big) = f'(x) \quad \mbox{for all} \quad x\in X\,.
$$
Moreover, $h$ is a homomorphism by Remark \ref{R:epi} because $h\cdot f=f'$ is  a homomorphism, and $f$  is a surjective homomorphism. Finally, $h$ is nonexpanding. Indeed, given $y_1$, $y_2\in Y$ of distance $\var$, we choose $x_i \in X$ with $y_i = f(x_i)$, and since $d\big( f(x_1), f(x_2)\big) =\var$, we conclude
$$
(x_1, x_2) \in D\var\,.
$$
Apply \thetag{C} above to this element to get
$$
d\big( f'(x_1), f'(x_2)\big)\leq\var
$$
which means 
$$
d\big( h(y_1), h(y_2)\big) =\var\,,
$$
as desired. Thus $h$ is a morphism of $\Sigma$-$\Met$ with $f' = h\cdot f$.
\end{proof}

\begin{remark}\label{R:tt}
	For every surjective morphism $f\colon M\to N$ in $\Met$ the morphism $f\otimes G\colon M\otimes G\to N\otimes G$ is a subregular epimorphism. Indeed, we have a diagram $D\colon \cd \to \Met$ with the colimit map $f$. The diagram $D\otimes G\colon \cb \to \ck$ given by tensoring with $G$ has the colimit map $f\otimes G$ because $-\otimes G$ is a left adjoint.
\end{remark}

We are going to define subcongruences as basic diagrams with properties analogous to cogruences in ordinary categories (Definition \ref{R:cong}). There is an additional continuity condition derived from the continuity of the assignment of $D\var$
to each $\var$.
 
	\begin{defi}\label{D:proeq} 
	Let $X$ be an object of a metric-enriched category $\ck$. A \emph{subcongruence} on $X$ is a basic diagram $D\colon \cb \to \ck$ with $Da = X$ such that  all pairs $Dl_{ \var}, Dr_\var$ are relations  (Definition \ref {D:iso}) satisfying
	the following conditions (for all numbers $\var$ and $\var  '$):
	
	(a) Reflexivity: $Dl_\var$, $Dr_\var \colon D\var \to X$ is $\var$-reflexive (Definition \ref{D:eps}).
	
	(b) Symmetry: $Dl_\var$, $Dr_\var \colon D\var \to X$ is symmetric.
	
	(c) Transitivity:  Given morphisms  $u_0, u_1, u_2 \colon S \to X$ such that 
	$(u_0, u_1)$ factorizes through $(Dl_\var , Dr_\var)$ and  $(u_1,u_2)$ factorizes through $(Dl_{\var'}, Dr_{\var'})$, then $(u_0, u_2)$ factorizes through $(Dl_{\var + \var '}, Dr_{\var + \var '})$.
	
	(d) Continuity: $D\var$ is the limit of the following $\omega^{\op}$-sequence in $\ck$:

	$$
	 D(\var+1) \leftarrow D(\var+1/2) \leftarrow D(\var+1/3)\cdots
	$$
	with the limit cone given by
	$$
	D(\var \to \var+1/n)\colon D\var \to D(\var +1/n) \qquad (n=1,2,3,\dots).
	$$
	
\end{defi}

\begin{remark}\label{R:pb}
Transitivity simplifies, whenever pullbacks exist, as follows:
 given the pullback $P$ of $Dr_\var$ and $Dl_{\var'}$ (on the left), there is a morphism $p$ making the diagram on the right commutative:
$$
\xymatrix@C=0.6pc@R=1.7pc{
	&&P\ar[dl]_{\bar l}\ar@{}[d]|{\vee} \ar[dr]^{\bar r}&&\\
	& D\var\ar[dl]_{Dl_\var} \ar[dr]_{Dr_\var}&& D\var' \ar [dl]^{Dl_{\var'}} \ar[dr]^{Dr_{\var'}}&\\
	X &&X&&X
}
\qquad
\xymatrix@C=0.2pc@R=1.7pc{
	&&P\ar	[dl]_{\bar l}\ar[dd]_p \ar [dr]^{\bar r}&&\\
	& D\var\ar[dl]_{Dl_\var}&& D\var'  \ar[dr]^{Dr_{\var'}}&\\
	X && D(\var+\var')\ar[ll]^{Dl_{\var+\var'}} \ar[rr]_<<<<<{Dr_{\var+\var'}}&&X
}
$$

\end{remark}

\begin{example} 
Every kernel diagram is a subcongruence. Indeed,  let $D$ be the kernel diagram  of $f\colon X\to Y$.

(1) The $\var$-kernel pair  $Dl_\var$, $Dr_\var$ of $f$ is $\var$-reflexive and symmetric (Remark \ref{R:ker}).

(2) Transitivity: If $u_0, u_1$ factorizes through the $\var$-kernel pair of $f$, then 
$d(f \cdot u_0 , f\cdot u_1) \leq \var$.  Analogously $d(f\cdot u_1 , f \cdot u_2) \leq \var'$. The triangle inequality yields $d(f \cdot u_0 , f \cdot u_2) \leq \var + \var'$, thus $(u_0, u_2)$
factorizes through the $(\var + \var')$-kernel pair of $f$, as claimed.

(3) Continuity: Let a cone of the above $\omega ^{op}$-sequence be given:
$$a_n \colon A \to D(\var + 1/n)$$
for $n=1, 2, 3...$. We present a morphism $a\colon A \to D\var$ with
$$a_n=D(\var \to \var + 1/n) \cdot a$$
for all n. The compatibility of the cone means that
$$a_{n+1} =D([\var + 1/n] \to [\var + 1/(n+1)]) \cdot a_n.$$
This implies that for all $n$ the composite of $a_n$ with $Dl_{\var + 1/n}$
is independent of $n$, and we denote this morphism by
$$\hat l = Dl_{\var + 1/n} \cdot a_n \colon A \to X.$$ Analogously we have
$$\hat r = Dr_{\var + 1/n} \cdot a_n \colon A \to X.$$
The distance of $f \cdot Dl_{\var + 1/n}$ and $f \cdot Dr_{\var + 1/n}$
is at most $\var + 1/n$, from which it follows that
$d(f \cdot \hat l, f \cdot \hat r) \leq \var + 1/n$ for each $n$.
In other words,
$$d(f \cdot \hat l, f \cdot \hat r) \leq \var.$$
Consequently the pair $(\hat l, \hat r)$ factorizes through the
$\var$-kernel pair of $f$: we get $a \colon A \to D\var$ with
$$ \hat l = Dl_\var \cdot a \quad and\ \hat r = Dr_\var \cdot a.$$
The desired equality $a_n = D(\var \to \var + 1/n) \cdot a$ now follows from
the fact that the pair $Dl_{\var + 1/n}, Dr_{\var + 1/n}$ is a relation,
thus collectively monic. We have
$$Dl_{\var + 1/n} \cdot a_n = \hat l =Dl_\var \cdot a=Dl_{\var + 1/n} \cdot
D(\var \to \var + 1/n) \cdot a.$$
Analogously for $Dr_{\var + 1/n}$. 
Hence, $D\var$ is the limit od $D(\var + 1/n).$

In order to verify that the above limit is conical, consider a parallel pair
$u_i \colon U \to A (i= 0,1)$. We are to show that its distance is the
infimum of the distances of $D(\var \to \var + 1/n) \cdot u_i$ for all
$n= 1,2,3...$. This follows easily from $\var = inf (\var + 1/n)$ and
the fact that the pair $Dl_\var, Dr_\var$ is isometric (Definition \ref{D:iso}).

\end{example}

\vspace{2mm}

In the following construction we use (extended) \textit{pseudometrics} (defined as metrics, except that distance $0$ between distinct elements is allowed).

\begin{remark} \label{R:pseudo}
For every pseudometric space $(X,d)$ the \textit{metric reflection}  is the quotient space
$$
f\colon (X,d) \to (X,d)/\sim
$$
 modulo the equivalence $x_1\sim x_2$ iff $d(x_1, x_2)=0$.
Observe that the quotient map $f$ preserves distances of  elements.
\end{remark}

\begin{odst}\label{C:eq} \textbf{Construction} of basic colimits of subcongruences $D\colon \cb \to \Met$. Put $Da = (M, d)$. The following defines a pseudometric $\hat d$ on the set $M$ 
$$
\hat d(x,y) = \inf \{\var; x = Dl_\var (t) \ \mbox{and}\ y= Dr_\var (t) \ \mbox{for some}\ t\in D\var\}\,.
$$
The colimit map
$$
f\colon (M,d) \to C
$$
of $D$ is the metric reflection of the pseudometric space $(M,\hat d)$: we have $C= (M, \hat d)/\sim$.
\end{odst}
Before proving this statement, we need an auxilliary fact about $\hat{d}$. The infimum defining $\hat d(x,y)$ is actually a minimum (whenever that distance is finite):

\begin{lemma}
\label{L:help}

Given elements $x, y$ of $M$ with $\hat d (x,y)=\var$, there exists  $t\in D\var$
with $x= Dl_\var (t)$ and $y= Dr_\var (t)$.
\end{lemma}
\begin{proof}
We denote by $S$ the set of all numbers $\delta \geq 0$ for which there exists $t \in D\delta$ with
\begin{equation} x = Dl_\delta (t) \quad and\  y= Dr_\delta(t). 
\tag{*}
\end{equation}
Observe that $t$ is unique since $ Dl_\delta,  Dr_\delta$ is a relation, thus a collectively monic pair. Our task is to prove that 
$\var$ lies in $S$.

The set $S$ is upwards closed: given $\delta \leq {\delta}'$, 
the element ${t}'= D(\delta \to {\delta)}' (t)$ fulfils \thetag{*}.
It follows that $S$ contains $\var + 1/n$ for all $n= 1,2,3 \cdots$. The corresponding elements $t_n$ of $D(\var + 1/n)$
are compatible with the $\omega ^{op}$-chain $D(\var + 1/n)$ with the connecting maps
$$f_n = D( [\var + 1/(n+1)] \to [\var + 1/n].$$
In fact, for each $n$ we have
$$t_n = f_n (t_{n+1}),$$
this follows from the uniqueness of $t$ in  \thetag{*}.

Limits in $\Met$ are $\Set$-based. By Continuity we have $D\var = lim D(\var + 1/n)$, 
thus for the compatible collection $t_n$
there exists a unique $t$ in $D\var$ with 
$$t_n = D( \var \to [\var + 1/n] )(t) \hspace{5mm} (n=1,2,3...).$$
In particular, $t_1 = D( \var \to [\var + 1] )(t)$.

The element $t$ fulfils  \thetag{*} for $\var$. Indeed the compatibility of the cone of all $Dl_\var$'s yields
$$Dl_\var = Dl_{\var  +1} \cdot D(\var \to [\var +1]),$$ 
therefore
$$Dl_\var (t)=Dl_{\var +1} (t_1) =x.$$
Analogously for $Dr_\var$. This proves that $\var$ lies in $S$, as claimed.  

\end{proof}

\begin{prop}
The morphism $f$ in the above Construction is the colimit map of the given subcongruence.
\end{prop}
\begin{proof}
(1) We first verify that $\hat d$ is a pseudometric.

(1a) $\hat d (x, x) =0$ because the pair $Dl_0$, $Dr_0$ is reflexive. Given $m\colon Da \to D0$ with $Dl_0 \cdot m=\id= Dr_0\cdot m$, put $t=m(x)$.
Then $x= Dl_0(t) = Dr_0(t)$.

(1b) The function  $\hat d$ is symmetric because all the pairs $Dl_\var$, $Dr_\var$ are symmetric.

(1c)  To verify the triangle inequality for $x$, $y$, $z\in M$, we prove
$$
\hat d(x,z) < \hat d(x,y) + \hat d(y,z) +\delta\quad \mbox{for each\quad} \delta>0\,.
$$
From the definition of $\hat d(x,y)$ it follows that there exists
$
\var <\hat d (x,y) + \frac{\delta}{2}
$
such that
$$
x = Dl_{\var} (t) \quad \mbox{and}\quad y= Dr_{\var}(t) \quad \mbox{for some}\quad  t\in D\var\,.
$$
Analogously there exists $\var' <\hat d (y,z) +\frac{\delta}{2}$ with
$$
y= Dl_{\var'} (s) \quad \mbox{and} \quad z  = Dr_{\var'}(s) \qquad \mbox{for some} \ \ s\in D\var'\,.
$$
Use the pullback $P$ in Remark \ref{R:pb}. Since
$$
Dr_{\var} (t) = y = Dl_{\var'} (s)\,,
$$
there exists $a\in P$ with
$$
t= \bar l (a) \quad \mbox{and} \quad s= \bar r(a)\,.
$$
The element
$$
b= p(a) \in D(\var + \var')
$$
then fulfils $x= Dl_{\var + \var'}$, since $Dl_{\var}\cdot \bar l= Dl_{\var+\var'} \cdot p$:
$$
x= Dl_{\var} (t) = Dl_{\var}\big(\bar l(a)\big) = Dl_{\var + \var'} \big(p(a)\big)=
Dl_{\var + \var'}(b)\,.
$$
Analogously,
$$
z= Dr_{\var + \var'}(b)\,.
$$
This proves  $ \hat d(x,z)\leq{\var + \var'}$, and we have
$$
\hat{d} (x,z) \leq\var +\var' <\hat d(x,y) +\frac{\delta}{2} + \hat d(y,z) +\frac{\delta}{2}\,,
$$
as required.

\vskip 1mm
(2) Put $C= (\bar M, \bar d)$. We next observe that 
the metric $d$ of $Da$ fulfils
$d\geq \hat d$, thus $f\colon (M,d) \to (\bar M, \bar d)$ is nonexpanding.

Given a pair  of elements $x_1, x_2\in M$ represented by morphisms  $u_1$, $u_2\colon 1\to (M,d)$ with $d(u_1, u_2) =\var$, there exists by $\var$-reflexivity a factorization $v\colon 1\to D\var$ with $u_1 = Dl_\var \cdot v$ and $u_2 = Dl_\var \cdot v$. In other words:   there exists $t\in D\var$ with $x_1 = Dl_\var (t)$ and $x_2 = Dr_\var (t)$. Which means $\hat d(x_1, x_2)\leq \var$. Thus $d(x_1, x_2) \geq \hat d(x_1, x_2)$.

\vskip 1mm
(3) We prove that the kernel diagram $D_f$ of $f$ is naturally isomorphic to $D$. The proof then follows from Lemma \ref{L:pro} since $f$ is a subregular epimorphism (Proposition \ref{P:sur}).

 Recall that the subspace $D_f \var$
consists of all pairs $x$, $y$ satisfying $d(f(x), f(y)) \leq \var$. That is, $d(x,y) \leq \var$, see Remark \ref{R:pseudo}. By the above lemma, these are the pairs
with $x=Dl_\var (t)$ and $y=Dr_\var (t)$ for some $t \in D\var$. Moreover, this $t$ is then uniqute, since $Dl_\var, Dr_\var$ is a 
relation (thus a collectively monic pair). This defines a natural transformation $u \colon D_f \to D$ whose $\var$-component is given by $(x, y) \mapsto t$. The inverse natural transformation from $D$ to $D_f$ has the $\var$-component
assigning to $t$ the pair $(Dl_\var (t), D(r_\var (t))$.
\end{proof}

\begin{prop}\label{P:eq} 
Colimits of subcongruences commute in $\Met$ with finite products. Explicitly:
let $D_i\colon\cb \to \Met$ be subcongruences  with colimit maps $f_i\colon D_i a \to C_i$ 
($i=1,2$). Then the basic diagram $D_1\times D_2\colon \cb\to \Met$ has the colimit map $f_1\times f_2$.
\end{prop}

\begin{proof}
Put $D=D_1\times D_2$ and 
$
D_ia = (M_i, d_i)$ for  $i=1,2.$
Then
$$ Da = (M_1\times M_2, d)
$$
where
$$
d(x,y) = \max \{d_1(x_1, y_1), d_2(x_2, y_2)\}\,.
$$
Using Construction \ref{C:eq}, it is sufficient to prove
$$
\hat d (x,y)= \max \{\hat d_1(x_1, y_1), \hat d_2(x_2, y_2)\} 
$$ 
to conclude that $\colim_B D$ is the product of $\colim_{B} D_i$ ($i=1,2$). 

We prove $\var =max (\var_1 , \var_2)$ where $\var = \hat{d} (x,y)$ and analogously $\var_i$ for $i=1, 2.$

(1) $\var \geq max(\var_1 ,\var_2)$.
This is clear if $\var$ is infinite. Otherwise Lemma \ref{L:help} yields $t=(t_1, t_2)$ in $M_1 \times M_2$
with $x= Dl_\var (t)$ and $Dr_\var(t)$. Thus for $i= 1,2$ we have $Dl_{\var_i} (t_i) = x_i$ and $Dr_{\var_i} (t_i) = y_i$.
This proves $\var_i \leq \var$.

(2)  $\var \leq max(\var_1 ,\var_2)$.   Without loss of generality we assume $\var_1 \leq \var_2 < \infty$. For $i=1,2$
we have $t_i \in D\var_i$ with 
$$x_i = Dl_{\var_i} (t_i) \quad \mbox{and} \ y_i = Dr_{\var_i} (t_i).$$
Let $t_1 '$ be the image of $t_1$ under $D(\var_1 \to \var_2)$. By compatibility  we get
$$x_1 = D_1 l_{\var_1} (t_1)= Dl_{\var_2} \cdot D(\var_1 \to \var_2) (t_1))= Dl_{\var_2} (t_1 ') $$
Thus the element $t=(t_1 ', t_2)$ of $D{\var_2}$ 
fulfils
$$(x_1, x_2) = (D_1 l_{\var_2} (t_1 '), D_2 l_{\var_2} (t_2)) = Dl_{\var_2} (t).$$
Analogously $(y_1, y_1)= Dlr_{\var_2} (t)$. This proves that
$$\var = \hat{d} ((x_1,x_2),(y_1, y_2) \leq \var_2 = max(\var_1, \var_2).$$

\end{proof}

\begin{remark}\label{R:free} 
For every space $M$ with metric $d_M$ the free $\Sigma$-algebra
$$
F_\Sigma M
$$
on $M$ was described in \cite{ADV} as follows. Its element are the $\Sigma$-terms with variables in $M$.
Its operations are, as usual, given by the formation of composite terms $t=\sigma(t_i)_{i<n}$ (for all $\sigma \in \Sigma_n$ and all $n$-tuples $(t_i)$ of terms). Terms $t$ and $t'$ are \textit{similar} if we can get $t'$ from $t$  by changing some variables of $t$. 
(Thus all pairs in $M$ are similar, and all similar composite terms have the same main symbol in $\Sigma$.) 
The metric of $F_\Sigma M$ assigns to terms $t$ and $s$ the distance $d(t,s)$ defined by structural induction (in the complexity of the terms) as follows

(1) $d(t,s) = d_M(t,s)$ if $t, s \in M$.

(2) $d(t,s)=\infty$  if $t$ and $s$ are not similar.

(3) $d(t,s)=\max\limits_{i<n} d(t_i, s_i)$ for similar composite terms $t=\sigma(t_i)_{i<n}$ and $s= \sigma(s_i)_{i<n}$ (where $\sigma \in \Sigma_n$).

The universal map $\eta_M \colon M \to U_\Sigma F_\Sigma M$ (Example \ref{E:alg}) is the inclusion map.
\end{remark} 

\begin{lemma}\label{L:free}
Let $A$ be a quantitative algebra. Given morphisms $f_1$, $f_2\colon M\to A$ in $\Met$ the corresponding homomorphisms $\bar f_1$, $\bar f_2 \colon F_\Sigma M \to A$ fulfil
$$
d(f_1, f_2) = d(\bar f_1, \bar  f_2)\,.
$$
\end{lemma}

\begin{proof}
We have $d(f_1, f_2) = d( \bar f_1\cdot \eta_M, \bar f_2\cdot \eta_M)\leq d(\bar f_1, \bar f_2)$. The opposite inequality follows from $d(f_1, f_2) \geq d(\bar f_1(t), \bar f_2(t))$ for all terms $t\in F_\Sigma M$,  which we verify by structural induction. Indeed, this is clear if $t\in M$. Given $t= \sigma (t_i)$ such that each $t_i$ satisfies  the above  inequality, then $\bar f_k(t) = \sigma_A \big(\bar f_k(t_i)\big)$ for $k=1,2$ implies that $t$ also satisfies that inequality, since $\sigma_A$ is nonexpanding.
\end{proof}

\begin{corollary}\label{C:free}
The functor $F_{\Sigma} \colon \Met \to \Sigma$-$\Met$ left adjoint to $U_\Sigma$ is locally nonexpanding.
\end{corollary}

\begin{remark}\label{R:pr}
The functor $U_\Sigma\colon \Sigma$-$\Met\to \Met$ creates products (in fact, all weighted limits). Given algebras $A_i$, $i\in I$, then the space $A = \Pi_{i\in I} (U_\Sigma  A_i)$  has, for each $\sigma \in \Sigma_n$, a unique operation $\sigma_A \colon A^n \to A$ preserved by all  projections ($\pi_i \cdot \sigma_A =\sigma_{A_i}\cdot \pi_i^n$ for all $i\in I$): given  an $n$-tuple  $f\colon n\to A$ put 
$$
\sigma_A(f) = \big(\sigma_{A_i}(\pi_i \cdot f)\big)_{i\in I}\,.
$$
Moreover, the resulting algebra is obviously  a product in $\Sigma$-$\Alg$.

The argument for weighted limits is similar.
\end{remark}

\begin{lemma}\label{L:co} 
$\Sigma$-$\Met$ has coequalizers.
\end{lemma}

\begin{proof}
Let $f_1, f_2\colon X\to Y$ be homomorphisms in $\Sigma$-$\Met$.  All surjective homomorphisms $e\colon Y\to Z_e$ in $\Sigma$-$\Met$ with $e\cdot f_1 = e\cdot f_2$ form a diagram $D$ (as full subcategory of $Y/ \Sigma$-$\Met$) with a limit given  by an object $Z$ and a cone $z_e\colon Z\to Z_e$. All $e$'s form a cone of $D$, thus we get a unique factorization 
$$
q \colon Y\to Z\quad \mbox{with}\quad e= z_e\cdot q\,.
$$
It is easy to see that $q \cdot f_1 = q\cdot f_2$, and $q$ is surjective. This is the coequalizer of $f_1, f_2$. Indeed, every homomorphism $h\colon Y\to Z$ with $h\cdot f_1= h\cdot f_2$  factorizes as $h=m\cdot e$ where $e$ is a surjective homomorphism and $m$ is an isometric embedding. Consequently, $e\cdot f_1 = e\cdot f_2$, hence $e$ (and thus $h$, too) factorizes through $q$.
\end{proof} 

\begin{defi}\label{D:eff} 
A metric-enriched category has \emph{effective subcongruences} if every 
subcongruence is a kernel diagram of some morphism.
\end{defi}

\begin{prop}\label{P:eff1}
The category $\Met$ has  effective subcongruences.
\end{prop}

\begin{proof}
Let $D\colon \cb \to \Met$ be a subcongruence with a colimit map $f\colon (M,d)\to (\bar M, \bar d)$ as in Construction \ref{C:eq}. We prove that $D$ is (naturally isomorphic to) the kernel diagram $\tilde D$ of $f$. For each $\var$ we have
$$
\hat d(x_1, x_2) =\var \quad \mbox{iff} \quad d\big(f(x_1), f(x_2)\big) = \var
$$
since $f$ preserves distances (Remark \ref{R:pseudo}). Thus $\tilde D\var$ is the subspace of $(M, d)^2$ on all pairs $(x_1, x_2)$ with $\hat d(x_1, x_2) \leq \var$. That is, all pairs $
\big(Dl_\var(t), dr_\var(t)\big)$ for $t\in D\var$: see Lemma \ref{L:help}.

Since $D$ is a  subcongruence, each pair  $Dl_\var$, $Dr_\var$ defines a subspace of $(Da)^2$ on the same set of pairs as above. The corresponding isomorphisms $u_\var \colon D\var \to \tilde D\var$ with  $Dl_\var = \tilde Dl_\var\cdot u_\var$ and $Dr_\var =\tilde Dr_\var \cdot u_\var$ are natural. Indeed, given $\var \leq \var'$  the naturality square below commutes
$$
\xymatrix@=3pc{ 
	D\var \ar[r]^{u_\var} \ar[d]_{D(\var\to \var')}
	& \tilde D\var \ar[d]^{\tilde D(\var \to \var')}&\\
	D\var' \ar[r]_{u_{\var'}} & \tilde D\var' \ar@<0.5ex>[r]^{\tilde Dr_\var'}
	\ar@<-0.5ex>[r]_{\tilde D l_{\var'}}& (M,d)
}
$$
due to the universality of the $\var'$-kernel pair $\tilde Dl_\var'$, $\tilde Dr_\var'$ and the fact that the above square commutes when post-composed by $\tilde Dr_\var'$ or $\tilde Dl_\var'$.
The naturality squares for $l_\var$ and $r_\var$ are obvious.
\end{proof}

\begin{prop}\label{C:cr} 
The forgetful functor $U_\Sigma \colon \Sigma$-$\Met \to \Met$  preserves subcongruences, and creates their colimits.
\end{prop}

Explicitly: if $D$ is a subcongruence in $\Sigma$-$\Met$, then $U_\Sigma D$ 
is a subcongruence in $\Met$. Given its 
colimit map
$$
c\colon U_\Sigma Da \to C\,,
$$
then $C$ carries a unique structure of a quantitative algebra making $c$ a homomorphism; moreover that homomorphism is the colimit map of $D$.

\begin{proof}
We first show that $U_\Sigma D$ is a subcongruence in $\Sigma$-$\Met$ 

(a) Reflexivity: Let 
$q_1$, $q_2 \colon Q\to U_\Sigma Da$ be morphisms of distance $\var$. Since the corresponding homomorphisms $\bar q_1$, $\bar q_2 \colon F_\Sigma Q\to Da$ also have distance $\var$ (Lemma \ref{L:free}), there exists a homomorphism $q\colon F_\Sigma  Q\to D\var$ with $\bar q_1 = Dl_\var \cdot  q$ and $\bar q_2 = Dr_\var \cdot  q$. Then the pair $(q_1, q_2)$ factorizes through $(UDl_\var, UDr_\var)$ via the morphism $q\cdot \eta_Q \colon Q \to U_\Sigma D\var$, using that $q_i = U_\Sigma \bar q_i \cdot \eta_Q$.

(b) Symmetry is clear.

(c) Transitivity follows from Remark \ref{R:pb} and the fact that $U_\Sigma$ preserves pullbacks.

(d) Continuity is clear since $U_\Sigma$ preserves limits.

Next we verify that $C$ above carries a unique  algebra structure making  $c$ a homomorphism. Given $\sigma \in \Sigma_n$, then $c^n$ is the colimit map of $(U_\Sigma D)^n= U_\Sigma D^n$ (Proposition \ref{P:eq}). The  morphism $c\cdot \sigma_{Da} \colon U_\Sigma D^na \to C$ fulfils the compatibility condition (COM) for  $U_\Sigma D^n$. Indeed, 
$Dl_\var$ and $Dr_\var$ are homomorphisms, therefore we have
\begin{align*}
d(c\cdot \sigma_{Da} \cdot U_\Sigma D^n l_\var, c \cdot \sigma_{Da} \cdot U_\Sigma D^n r_\var) &\leq  d (\sigma_{Da} \cdot U_\Sigma D^n l_\var, \sigma_{Da} \cdot U_\Sigma D^n r_\var)\\[3pt]
&= d( U_\Sigma Dl_\var \cdot \sigma_{D\var},  U_\Sigma Dr_\var \cdot \sigma_{D\var})\\[3pt]
&\leq  d(U_\Sigma D l_\var, U_\Sigma Dr_\var)\\[3pt]
&\leq \var\,.
\end{align*}
Thus there is a unique morphism $\sigma_C \colon C^n \to C$ with
$$ 
c\cdot \sigma_{Da} = \sigma_C \cdot c^n\,.
$$
The verification that the resulting homomorphism $c$ is a colimit map of $D$  in $\Sigma$-$\Met$ is easy.
\end{proof}
From this proposition we get an immediate

\begin{corollary}\label{C:iff}
The category $\Sigma$-$\Met$ has effective subcongruences.
\end{corollary}

\section{Subvarietal Generators}\label{sec4}

 In Section \ref{sec2} we have recalled the concept of a varietal generator in an ordinary category: an abstractly finite, regularly projective regular generator. We now introduce the analogous concept in a metric-enriched category, and prove that a category with such a generator and reflexive
 coequalizers has weighted limits and colimits. This is used in the subsequent section for proving that, whenever
 subcongruences are, moreover, effective, the category is equivalent to a variety.

 Recall that $\ck$ denotes a category enriched over $\Met$, and that subregular epimorphisms in $\Met$ are precisely the surjective morphisms (Proposition \ref{P:sur}).
 
 \begin{defi}\label{D:pr} 
An object $K$ of  $\ck$ is a \emph{subregular projective} if its hom-functor $\ck(K, -) \colon \ck \to \Met$ preserves subregular epimorphisms. 
 Explicitly: for every subregular epimorphism $e\colon X \to Y$ all morphisms from $K$ to $Y$ factorize through $e$.
 
 \end{defi}
 It follows from Lemma \ref{L:pro} that in a category with $\var$-kernel pairs we have
 \begin{center}
 subeffective \ $\Rightarrow$ \ subregular projective.
 \end{center}

 \begin{example} 
 In $\Met$ subregular projectives are precisely the discrete spaces. Indeed:
 
 (1) If $K$ is a discrete space of cardinality $n$, then its hom-functor is naturally isomorphic to $(-)^n$. This functor preserves surjective morphisms, that is, subregular epimorphisms (Proposition \ref{P:sur}).
 
 (2) If $K$ is a subregular projective on a set $|K|$, use that the identity-carried morphism $e\colon |K| \to K$ is a subregular epimorphism
(Example \ref{E:bas}). Thus $\id_K$ factorizes through $e$ which means that $e$ is an isomorphism.

 \end{example}
 
 \begin{lemma} \label{L:fr}
 The free algebra $F_\Sigma M$ on a finite discrete space $M$ is a subregular projective in $\Sigma$-$\Met$.
  \end{lemma}
 
\begin{proof}
 If $M$ has $n$ elements, then its hom-functor is naturally isomorphism to $U_{\Sigma}^n$. We know that $U_\Sigma$ creates colimits of subcongruences (Corollary \ref{C:cr}) and that in $\Met$ colimits of subcongruences commute with finite products (Proposition \ref{P:eq}). Thus $U_{\Sigma}^n$ preserves colimits of subcongruences.
\end{proof}


\begin{lemma}\label{L:ep}
Let $G$ be an object with tensors (Example \ref{E:ten}) in $\ck$.
For every surjective morphism $h\colon M\to M'$ of $\Met$ the morphism $h\otimes G$ is a subregular epimorphism of $\ck$.
\end{lemma}

\begin{proof}
 Since $-\otimes G$ is left adjoint to $\ck(G, -)$, it preserves weighted colimits. Let $h$ be the colimit map of $D\colon \cb \to \Met$ weighted by $B$ (Proposition \ref{P:sur}). Then the colimit map of the composite $D\otimes G$ of $D$ and $-\otimes G$ (weighted by $B$) is $h\otimes G$.
 \end{proof}

%
 

 Recall the concept of abstract finiteness from Definition \ref{D:abs}. In the category $\Met$ the only finitely generated object is the empty space (\cite{AR3}).
 In contrast:
 
 \begin{example}\label{E:abs}
 	(1) A metric space is abstractly finite in $\Met$ iff it has only finitely many connected components. Thus finite spaces are abstractly finite, and so is for example $\R$ with the usual metric.
 	
 	\vskip 1mm
 	(2) The free algebra $F_\Sigma I$ on one generator, $x$,  is abstractly finite in $\Sigma$-$\Met$. Indeed, a coproduct $\coprod\limits_{M} F_\Sigma I$ is the free algebra on the discrete space $M$. Every morphism $f\colon F_\Sigma I\to F_\Sigma M$ is determined by the term $t= f(x)$ in $F_\Sigma M$.
 	This term contains only a finite subset $M_0 \subseteq M$ of variables. This implies that $f$ factorizes through the canonical embedding $F_\Sigma M_0 \hookrightarrow F_\Sigma M$.
 \end{example}
 
Recall that in an ordinary category an object $G$ with copowers is a regular generator provided that for every object $K$
the canonical morphism from the copower of $G$ indexed by $\ck (G,K)$ to $K$ is a regular epimorphism. Here is the 
metric-enriched variant:

\begin{defi}\label{D:vg}
Let $G$ be an object with tensors. It is a \textit{subregular generator} if
for every object $K$  the evaluation map $\hat\id \colon \ck(G,K)\otimes G\to K$ (Example \ref{E:n}) is a subregular epimorphism. 

$G$ is a \emph{subvarietal generator} if it is, moreover, an abstractly finite subregular projective.

\end{defi}



\begin{example}\label{E:reg}
The free algebra on one generator, $G=F_\Sigma I$, is a subvarietal generator of $\Sigma$-$\Met$:
it is abstractly finite (Example \ref{E:abs}) and a subregular projective (Lemma \ref{L:fr}). It also is
a subregular generator. Indeed, for every space $M$ we have
$$
F_\Sigma M = M \otimes F_\Sigma I
$$
because each quantitative algebra $A$ yields natural isomorphisms
$$
\xymatrix@C=.5pc@R=.21pc{
&F_\Sigma M\ar[r] & A&\\
\ar@{-}[rrr]&&&\\
&M  \ar[r] &U_\Sigma A&\\
\ar@{-}[rrr]&&&\\
 & M \ar[r] &\Sigma\mbox{-} \Met (F_\Sigma I, A)&
 }
$$
Thus, the evaluation morphism 
$$
\hat \id \colon \Sigma\mbox{-}\Met(F_\Sigma I, A) \otimes F_\Sigma I\to A
$$
 is the unique extension of $\id_{U_\Sigma A}$ to a  homomorphism from $F_{\Sigma} U_{\Sigma} A$ to $A$. This homomorphism 
is surjective, thus, a subregular epimorphism (Corollary \ref{C:var-epi}).

\end{example}

 \begin{remark}\label{R:faith}
If $G$ is a subregular generator, for every parallel pair $p, p' \colon P\to R$ we have
 $$
 d(p, p')= \sup_{t\colon G\to P} d(p\cdot t, p'\cdot t)\,.
 $$
 Indeed, this means precisely that the morphism $[g] \colon \coprod\limits_{|\ck (G,P)|} G \to P$ is a quotient.
 \end{remark}


 \begin{defi}[{\cite{K}}] 
 Let $\ca$ be a full enriched subcategory of $\ck$ and  denote by $E\colon \ck \to [\ca^{\op}, \Met]$ the functor assigning to $K$  the restriction of $\ck(-, K)$ to $\ca^{\op}$. Then $\ca$ is \textit{dense} if $E$ is  faithful (for all parallel pairs $f$, $f'$ in $\ck$ we have $d(f, f') = d(Ef, Ef')$) and full.
 \end{defi}
 
 \begin{theorem}\label{T:cc}
 Let $\ck$ have reflexive  coequalizers and a subvarietal generator $G$. Then  all finite copowers of $G$ form a dense subcategory.
 \end{theorem}
 
 \begin{proof}
  
   Denote by $\ca$ the full subcategory of all finite copowers
  $$
  n\otimes G = \coprod_n G \qquad (n \in \N)\,.
  $$
  We prove that $E$ in the preceding definition is full and faithful.
  
  (1)
  $E$ is faithful due to Remark \ref{R:faith}.
  
  (2) To prove that $E$ is full, consider a morphism from $EK$ to $EL$ for some $K$, $L \in \ck$, i.e.\ a natural transformation from $\ck(-, K) / \ca^{\op}$ to $\ck(-, L)/ \ca^{\op}$. We denote it as follows:
  $$
  \xymatrix@R=.21pc{
  	&n\otimes G \ar[r]^{\quad f\quad } & K&\\
  	\ar@{-}[rrr]&&&\\
  	&n\otimes G \ar[r]_{f'} & L&
  }\qquad (n\in \N)
  $$
  for all $n\in \N$ (considered as the discrete spaces $\{0, \dots, n-1\}$). The naturality of $(-)'$ means that the following implication holds for every morphism $h\colon n\otimes G \to m\otimes G$ in $\ca ^{op}$:
  \begin{equation}\label{in0}
  \xymatrix@C=1pc{
  	n\otimes G \ar[rd]_{f} \ar[rr]^{h} \ar@{}[drr]|{\circlearrowright}&& m\otimes G\ar[ld]^{g}\\
  	& K &
  }
  \quad 
  \Rightarrow
  \quad 
  \xymatrix@C=1pc{
  	n\otimes G  \ar[dr]_{f'} \ar[rr]^{h} \ar@{}[drr]|{\circlearrowright}&& m\otimes G\ar[dl]^{ g'}\\
  	& L &
  }
  \end{equation}
  Applying this to $n=1$ and the coproduct injections $h_i \colon G \to m$ of $m\otimes G$, we conclude for each morphism $g=[g_i]$ that
  \begin{equation}\label{ine11}
  [g'_0 ,...,g'_{m-1}]=[g_0, ...,g_{m-1}]'
  \end{equation}
  
  (2a) We are to present a morphism
  \begin{equation}\label{equ}
  k\colon K \to L \quad \mbox{with}\quad f'= k\cdot f \qquad \mbox{(for all $f\colon n \otimes G \to L$)}.
  \end{equation}
  For that we denote the adjoint transpose of $(-)' \colon \ck (G, K,) \to \ck (G, L)$ by
  \begin{equation}
  q \colon \ck (G,K) \otimes G \to  L.
  \end{equation}
  That is, $q$ is given by the property that for every morphism $f \colon G \to K$ represented by
  $f_0 \colon 1 \to \ck (G,K)$ (thus $f=\hat{f_0})$ we have
  $$f'=q  \cdot (f_0 \otimes G).$$ 
  It follows from \eqref{ine11} that for every finite (discerete) subset $m \colon M \hookrightarrow \ck (G,K)$ we get
  for $\hat m \colon M \otimes G \to K$ that
  \begin{equation}\label{nnew}
  \hat{m'}=q \cdot (m\otimes G).
  \end{equation}
  We verify below that $q$ factorizes through $\hat{\id}$. This conludes the proof, since the factorizing morphism
  $k$:
  $$
  \xymatrix@=4pc{
  	\ck (G,K)\oti G \ar[r]^<<<<<<<<{q}\ar[d]_{\hat \id}& L\\
  	K \ar[ur]_k & G \ar[l]^{f} \ar[u]_{f'}
  }
  $$
  fulfils \eqref{equ}. Indeed, given $f\colon G\to K$, we have $f= \widehat{\id} \cdot(f_0 \otimes G)$, 
  thus 
  $$f'=q  \cdot (f_0 \otimes G)=k \cdot \widehat{\id}  \cdot (f_0 \otimes G) = k \cdot f.$$
  
  (2b) We now prove that $q$ factorizes through $\widehat \id$. Since $\widehat \id $ is a subregular epimorphism, we have a diagram $D \colon \cb \to \ck$ with
  $$
  K=\colim_{B} D\,, \qquad  \ck (G, K) \oti G = Da,
  $$
  and the colimit map  $\widehat \id$. It is sufficient to verify that for every $\var$ the morphism $q$ satisfies  the compability condition \thetag{COM}:
  \begin{equation}\label{ine1}
  d(q\cdot D{l_\var}, q\cdot Dr_{\var}) \leq \var\,.
  \end{equation}
By Remark \ref{R:faith} this is equivalent to proving that given $t\colon G\to D\var$ we have
  \begin{equation}\label{ine2}
  d(q\cdot Dl_\var\cdot t, q\cdot Dr_{\var}\cdot t) \leq \var\,.
  \end{equation}
  We fix $\var$ and $t$ and prove this inequality now.
  
  The identity-carried map
  $$ 
  b\colon |\ck (G,K)| \to \ck(G,K)
  $$
  is by Proposition \ref{P:sur}  a subregular epimorphism in $\Met$. Thus, so is $b \otimes G$ (Remark \ref{R:tt}).
  Since $G$ is subregularly projective, both morphisms
  $D l_\var \cdot t$, $Dr_\var \cdot t \colon G \to \ck (G,K) \oti G$  factorize through 
  
$$  b\otimes G \colon \coprod_{|\ck(G,K)|} G \to \ck(G,K) \oti G\,.
  $$
  Moreover, since $G$ is abstractly finite, there is a finite subset $m\colon M\hookrightarrow |\ck (G,K)|$ such that both of the factorizing morphisms further factorize through the sub-coproduct $m \otimes G \colon \coprod\limits_M G \hookrightarrow |\ck (G,K)|$. We denote the resulting factorizing morphisms by $\overline{l_\var}$ and $\bar r_\var$, respectively: 
  $$
  \xymatrix@C=2pc@R=4.5pc{
  	G\ar[rr]^{t}
  	\ar @<0.6ex>[d]^{\overline{r}_\var}
  	\ar@<-0.6ex> [d]_{\overline{l}_\var}
  	&&
  	D\var
  	\ar @<0.6ex>[d]^{Dr_\var}
  	\ar@<-0.6ex> [d]_{Dl_\var} & \\
  	\coprod\limits_{M} G \ar[r]^<<<<<{m\otimes G}  \ar[drr]_{\widehat{b\cdot m}}& \!\!\!\coprod\limits_{|\ck(G,K)|}G\!\!\! \ar[r]^<<<<<{b\otimes G} &  \ck(G,K) \otimes G \ar [r]_>>>>>>{q} \ar[d]^{\widehat\id}& L\\
  	&&K&&}
  $$
  In the above diagram the lower tringle commutes because $\widehat{b\cdot m}=\hat{id} \cdot [(b \cdot m) \otimes G]$.
  Since $\hat\id$ satisfies the compatibility condition $d(\hat\id \cdot Dl_\var, \hat\id \cdot D r_\var)\leq \var$, we get
  $$
  d(\hat\id \cdot Dl_\var\cdot t, \hat\id \cdot Dr_\var \cdot t) \leq \var\,.
  $$
  We conclude  from the above diagram that
  $$
  d\big( \widehat{b\cdot m} \cdot \overline{l}_\var,  \widehat{b\cdot m}\cdot \overline{r}_\var \big) \leq \var\,.
  $$
  Applying \eqref{in0} we thus get
  $$
  d\big([ \widehat{b\cdot m}]' \cdot \overline{l}_\var,  [\widehat{b\cdot m}]'\cdot \overline{r}_\var \big) \leq \var\,.
  $$
  By the definition of $q$, we have $([ \widehat{b\cdot m}]'= q\cdot [(b\cdot m) \oti G] $. The desired inequality \eqref{ine2}
  follows from the above diagram, using the last inequality.

\end{proof}

\begin{corollary}\label{C:cc} 
Every category with reflexive coequalizers and a subvarietal generator $G$ has weighted limits and colimits.
\end{corollary}

$\Met$ has weighted limits and colimits (\cite[Example 4.5]{AR3}). 
By Corollary 3.5 in \cite{A1} a metric-enriched category with reflexive coequalizers has weighted limits and colimits, whenever
it has an object with tensors such that finite copowers of it are dense.

\begin{prop}\label{P:fac}
Every category with reflexive coequalizers and a subvarietal generator has the following factorization system $(\ce, \cm)$:
\begin{align*}
\ce & = \mbox{subregular epimorphisms,}\\
\cm & = \mbox{isometric embeddings (Definition \ref{D:iso})}\,.
\end{align*}
\end{prop}
 
 \begin{proof}
 (1) Existence of factorizations. Given a morphism $f\colon X\to Y$, we form its   kernel diagram (Definition \ref{D:ker}) $D\colon \cb \to \ck$. Let $e\colon X\to Z$ be the colimit map of $D$.
 
$$
\xymatrix@C=1pc@R=2.5pc{
& P_\var \ar@<0.5ex>[d]_{l_\var\  }
\ar @<-0.5ex>[d]^{\ r_\var} & & \\
&X \ar[rr]^f \ar[dr]_{e}&& Y\\
G\ar@<0.5ex>[rr]^{u_2}
\ar@<-0.5ex>[rr]_{u_1}
\ar@<0.5ex>[ur]^>>>>>>{v_2}
\ar@<-0.5ex>[ur]_>>>>>>{v_1}
\ar@{-->}@/^1pc/[uur]^t&& Z \ar[ur]_m&
}
$$
Then $f$ factorizes as $f= m\cdot e$. We know that $e$ is a subregular epimorphism. Let us prove that $m$ is an isometric embedding. Since $G$ is a subregular  generator, we just need to prove, using Remark \ref{R:faith}, that  for all pairs $u_1$, $u_2\colon G\to Z$ we have
$$
d(u_1, u_2)\leq d(m\cdot u_1, m\cdot u_2)\,.
$$

We know that $G$ is a subregular projective, hence $u_i$ factorizes as $u_i = e\cdot v_i$ ($ i=1,2$).
Put
$$
\var = d(m\cdot u_1, m\cdot u_2)\,.
$$
Then
$$
d(f\cdot v_1, f\cdot v_2) = d(m\cdot u_1, m\cdot u_2)\leq\var\,.
$$
Thus there exists $t\colon G\to D\var$ with $v_1 = l_\var\cdot t$ and $v_2= r_\var\cdot  t$. From $d( e\cdot Dl_\var, e\cdot Dr_\var)\leq \var $ we obtain
$$
d(u_1, u_2) = d(e\cdot Dl_\var\cdot t, e\cdot Dr_\var \cdot t) \leq d(e\cdot Dl_\var, e\cdot Dr_\var)\leq \var\,.
$$
We have proved  $d(u_1, u_2)\leq d(m\cdot u_1, m\cdot u_2)$. 

(2) The diagonal fill-in. Let a commutative square be given with $e$ a subregular epimorphism and $m$ an isometric embedding:
$$
\xymatrix@C=2pc@R=2pc{
D\var \ar@<0.5ex>[r]^{Dr_\var}
      \ar@<-0.5ex>[r]_{Dl_\var} & X \ar[r]^e \ar[d]_u &Y \ar[d]^v\\
& A\ar[r]_m& B
}
$$
Form the kernel diagram $D\colon \cb \to \ck$ of $e$ with colimit map $e$ (Lemma \ref{L:pro}). Since $d ( e\cdot Dl_\var, e\cdot Dr_\var)\leq \var$  implies $d(v\cdot e\cdot Dl_\var, v\cdot e\cdot r_\var)\leq \var$,  and 
 $m$ is an  isometric embedding, we get that 
$$
d(u\cdot Dl_\var , u\cdot Dr_\var)\leq \var \qquad \mbox{(for all} \ \var)\,.
$$
By the universal property of $e$ we obtain a morphism
$$
d\colon Y\to A \quad \mbox{with}\quad u=d\cdot e\,.
$$
Since $e$ is epic by Remark \ref{R:pro}, we conclude that $d$ is the desired diagonal: we have $v=m\cdot d$, and $d$ is unique.
\end{proof}

 In the definition \ref{D:abs}of abstract finiteness of $G$ only copowers of $G$ play a role.  The stronger (and more natural) property would consider morphisms from $G$  to arbitrary coproducts. For subvarietal generators this makes no difference:

 \begin{lemma}\label{L:abs} 
 Let $\ck$ have coproducts. Each morphism from a subvarietal generator $G$ to an arbitrary  coproduct
factorizes through a finite subcoproduct.
 \end{lemma}
 
 \begin{proof}
 First observe that a coproduct of subregular epimorphisms 
$$
e_i\colon X_i \to M_i \qquad (i\in I)
$$
 is a subregular epimorphism. Indeed, given basic  diagrams $D_i\colon \cb \to \ck$ with colimit maps $e_i$ ($i\in I$), then the diagram $D\colon \cb\to \ck$ defined by $D(-) = \coprod\limits_{i \in I} D_i(-)$ has the colimit map  $ \coprod\limits_{i \in I} e_i$, since weighted colimits commute with coproducts..
 
 Given  a morhism $f\colon G\to   \coprod\limits_{i \in I} C_i$, we use the following subregular epimorphisms 
$$
v =  \coprod\limits_{i\in I} \widehat \id \colon \coprod\limits_{i\in I} \ck (G, C_i)\oti G\to \coprod\limits_{i\in I}C_i
$$
 (Example \ref{E:n}) and 
$$
w\colon \coprod\limits_{i\in I} |\ck (G, C_i)| \oti G\to\coprod\limits_{i\in I} \ck (G, C_i) \oti G\,.
$$
 The latter is the coproduct of $w_i\oti G$ for the morphisms $w_i\colon |\ck (G, C_i)| \to \ck(G, C_i)$ carried by identity. 
 Since $G$ is subregularly projective, $f $ factorizes through $w\cdot v$: we have 
 $
 f'\colon G \to \coprod\limits_{i\in I} \ck (G, C_i)| \oti G $
 making  the following square commutative:
$$
\xymatrix@=3pc{ 
G\ar[r]^{f} \ar[d]_{f'} & \coprod C_i\\
\coprod |\ck(G, C_i)|\otimes G \ar[r]_{w} & \coprod \ck (G, C_i) \otimes G
\ar[u]_{v}
}
$$
 The codomain of $f'$ is a copower of $G$, thus there is a finite set $M\subseteq \coprod\limits_{i\in I} |\ck (G, C_i)|$ such that $f'$ factorizes through the corresponding subcopower. We have a finite set $J\subseteq I$ with $M\subseteq \coprod\limits_{j\in J} |\ck (G, C_i)|$. It follows easily from $f= f'\cdot w \cdot v$ that  $f$ factorizes through the subcoproduct $\coprod\limits_{j\in J} C_j$.
 \end{proof}

 \section{Varieties of Quantitative Algebras}\label{sec5}
 
 Here our main result is proved: a characterization of categories equivalent to varieties of quantitative algebras. Recall the free algebras $F_\Sigma M$ from Remark \ref{R:free}. We will use terms in $T_\Sigma X = U_\Sigma F_\Sigma  X$, where $X$ is a finite set (or discrete space). Recall further  that $\var$ denotes non-negative reals.
 
 \begin{defi}[{\cite{MPP17}}]
 (1) A \emph{quantitative equation} is an expression
 $$
 t=_\var t'
 $$
 where $t$ and $t'$ are terms in $T_\Sigma X$ for a finite set $X$ (of variables) and $\var \geq 0$ is a real number.
 
 (2) An algebra $A$ \emph{satisfies} $t=_\var t'$ if for every interpretation of the variables $f\colon X\to A$ the corresponding homomorphism $\bar f \colon T_\Sigma X \to A$ fulfils
 $$
 d\big(\bar f(t), \bar f(t')\big) \leq \var\,.
 $$
 
 (3) A \emph{variety} (aka 1-basic variety) of quantitative algebras is a full subcategory of $\Sigma$-$\Met$ specified by a set  of quantitative equations.
   \end{defi}
   
\begin{exs} (1) \emph{Quantitative monoids} are monoids acting on a metric space $M$ such that the multiplication is a nonexpanding map from $M\times M$ (with the maximum metric) to $M$. This is a variety presented by the usual signature (of one binary symbol and one constant $e$) and the usual equations:
$$
(xy) z =_0 x(yz)\,, \quad xe =_0 x \quad \mbox{and}\quad ex=_0 x\,.
$$

(2) \emph{$\var$-commutative monoids} are quantitative monoids which satisfy $d(ab, ba) \leq \var$ for all elements $a$, $b$. They are specified by the above equations plus
$$
xy =_\var yx\,.
$$

(3) The number $\var$ above is required to be rational in \cite{MPP17}, rather than real. But this is unimportant. An equation $t=_\delta t'$ where $\delta$ is irrational can be simply substituted by quantitative equations $t=_{\var_n} t'$ ($n\in \N$) for an arbitrary decreasing sequence of rational numbers $\var_n\geq \delta$ with $\lim\limits_{n\to \infty} \var_n =\delta$.
\end{exs}

\begin{defi}
 Let $A$ be an algebra in $\Sigma$-$\Met$.

(1) By a \emph{subalgebra} is meant a metric subspace closed under the operations. More precisely: subalgebras of $A$ are subobjects represented by morphisms $m\colon B\to \ca$ of $\Sigma$-$\Met$ carried by isometric embeddings (Definition \ref{D:iso}).

(2) By a 
\emph{homomorphic image} of $A$ is meant a quotient object represented by a surjective homomorphism $e\colon A\to B$ in $\Sigma$-$\Met$. Those are the subregular epimorphisms (Proposition \ref{P:sur}). 
\end{defi}

The following theorem was stated in \cite{MPP17}. The proof there is incomplete, for a complete proof see B19--B20 in \cite{MU} or  Thm. 5.10 in \cite{R}.

\begin{birk}\label{T:BVT}
 A full subcategory of $\Sigma$-$\Met$ is a variety of quantitative algebras iff it is closed under products, subalgebras, and homomorphic images.
 \end{birk}
 
 \begin{corollary}\label{C:var-epi}
 Subregular  epimorphisms in a variety $\cv$ are precisely the surjective morphisms.
 \end{corollary}
 
 Indeed, this is true for $\cv =\Sigma$-$\Met$ (Proposition \ref{P:sur}). Given a surjective morphism $f\colon X\to Y$ in $\cv$  we thus know that $f$ is the colimit morphism of its kernel diagram $D\colon \cb \to \Sigma$-$\Met$ (Lemma \ref{L:pro}). Each $D\var$ is a subalgebra of $X\times X$, thus $D\var \in \cv$. So $D$ is a diagram in $\cv$ with the colimit map $f$ in $\Sigma$-$\Str$ which lies in $\cv$. It follows that $f$ is a colimit map of $D$ in $\cv$. The converse implication is proved  as Proposition \ref{P:sur}.
 
 \begin{prop}[{\cite{MPP17}}] 
Every variety $\cv$ has free algebras: the forgetful functor from $\cv$ to $\Met$ has an enriched left adjoint $F_{\cv} \colon \Met \to \cv$.
\end{prop}

 Indeed, this holds for $F_\Sigma$ (Corollary \ref{C:free}). The existence of free algerbras was proved in \cite{MPP17}. The canonical natural transformation $\varrho \colon F_\Sigma \to F_{\cv}$ has surjective components $\varrho_X$ (using \ref{T:BVT}), thus they are quotients. This easily implies that $F_{\cv}$ is enriched.

\begin{remark}\label{R:effvar}
	Varieties have effective subcongruences. Indeed,  we know that $\Sigma$-$\Met$ does (Corollary \ref{C:iff}). Every subcongruence $D$ in $\cv$ is also  a subcongruence in $\Sigma$-$\Met$, since $\cv$ is closed under pullbacks (Remark \ref{R:pb}). Let $f\colon Da \to C$ be a colimit map of $D$ in $\Sigma$-$\Met$. It is a surjective homomorphism (Corollary \ref{C:var-epi}) whose kernel diagram is $D$ (Lemma \ref{L:pro}). Since $\cv$ is closed under homomorphic images, we conclude $C\in \cv$. Thus $f$ is a morphism of $\cv$ with the kernel diagram $D$.
\end{remark}

Recall subvarietal generator (Definition \ref{D:vg}).
 \begin{prop}\label{P:ge} 
 The free algebra $G$ on one generator in a variety $\cv$ is a subvarietal generator.
 \end{prop}
 
  \begin{proof}
 (1) $G$ is an abstractly finite subregular subgenerator: the argument is completely analogous to Example \ref{E:reg}.
 
 (2) $G$ is a subregular projective. Indeed, its hom-functor is naturally isomorpphic to the forgetful functor  $U_\cv\colon \cv \to \Met$ of $\cv$ which  preserves colimits of subcongruences because  the forgetful functor $U_\Sigma$ of $\Sigma$-$\Met$ creates them (Corollary \ref{C:cr}), and we have
  $$
  U_\cv = U_{\Sigma}\cdot E
  $$
  for the embedding $E\colon \cv \hookrightarrow \Sigma$-$\Met$. Now $E$ preserves colimits of subcongruences because the colimit morphisms in 
 $\Sigma$-$\Met$ are surjective and $\cv$ is closed under homomorphism images (Theorem \ref{T:BVT}). The surjectivity of colimit maps in  $\Sigma$-$\Met$ follows from $U_\Sigma$ creating colimits of subcongruences, and from the description of those colimits in $\Met$ (Construction \ref{C:eq}).
 \end{proof}
 
 \begin{corollary}\label{C:var} 
 Every variety of quantitative algebras has weighted limits and colimits.
 Its forgetful functor to $\Met$ creates limits.
  \end{corollary}
 Creation of limits follows from Remark \ref{R:pr} and the fact that, due to Theorem \ref{T:BVT}, the variety $\cv \subseteq \Sigma$-$\Met$ is closed under limits. The existence of weighted colimits follows from  Corollary \ref{C:cc}. We just need to observe that the variety $\cv$ has reflexive coequalizers because $\Sigma$-$\Met$ does (Lemma \ref{L:co}), and
 $\cv$ is closed under homomorphic images. 
  

\begin{theorem}\label{T:main} 
A metric-enriched category is equivalent to a variety of quantitative algebras iff it has
reflexive coequalizers, effective subcongruences and a subvarietal generator.
\end{theorem}

\begin{proof}
 Necessity follows from the above corollary, Remark \ref{R:effvar}, and Proposition \ref{P:ge}.

Sufficiency: let $G$ be a subvarietal generator of an enriched category $\ck$. 
Define a signature $\Sigma$ by
$$
\Sigma_n = |\ck (G, \coprod_{n} G) |\qquad (n\in \N)\,.
$$
Thus, an $n$-ary operation symbol is a morphism $f\colon G\to G +\dots + G$ ($n$ summands). We obtain a functor
$$
E\colon \ck \to \Sigma\mbox{-}\Met
$$
assigning to every object $K$ the algebra $EK$ on the metric space  $\ck(G,K)$ whose operation $\sigma_{EK}$ for $\sigma\in \Sigma_n$ we specify next. For each $n$-tuple $x_i\colon G\to K$ we define the result of $\sigma_{EK}$ in $\ck (G,K)$ as the following composite
$$
\sigma_{EK} (x_i)_{i<n} \equiv G\xrightarrow{\sigma} \coprod\limits_n G \xrightarrow{[x_i]} K\,.
$$
This is a nonexpanding map $\sigma_{EK} \colon (EK)^n \to EK$ because by Corollary \ref{C:cc} coproducts are conical.
To every morphism $f\colon K\to L$ the functor $E$  assigns the map post-composing with $f$:
$$
Ef = \ck (f,K) \colon \ck (G,K) \to \ck (G,L)\,.
$$

(a) $E$ is a well defined, fully faithful functor. Indeed, each $Ef$ is clearly nonexpanding and  a homomorphism: given $\sigma \in \Sigma_n$ and an $n$-tuple  $x_i \colon G\to K$, then
\begin{align*}
Ef \big(\sigma_K(x_i)\big) &= f\cdot [x_i] \cdot \sigma\\
&= [f\cdot x_i] \cdot \sigma\\
 &= \sigma_L \big(Ef (x_i)\big).
 \end{align*}
Thus $E$ defines a functor.
It is locally nonexpanding by Remark \ref{R:faith}.
 Moreover the  triangle below commutes:
$$
\xymatrix@C=1pc{
\ck \ar [rr]^E \ar[dr]_{\ck(G,-)} && \Sigma\mbox{-}\Met \ar[dl]^{U_\Sigma}\\
& \Met&
}
$$
Thus $E$ is faithful because $\ck(G,-)$ is due to Remark \ref{R:faith}. It is full because all finite copowers $\coprod\limits_{n}G$ are dense (Theorem \ref{T:cc}). Thus, given a homomorphism $k\colon EK\to EL$, we obtain a natural transformation from $E(-, K)$ to $E(-, L)$ as follows:
$$
\xymatrix@R=.21pc{
&\coprod_n G \ar[r]^>>>>>{[x_i] } & K&\\
\ar@{-}[rrr]&&&\\
&\coprod_n G \ar[r]_<<<<{[k(x_i)]} & L&
}
$$
The corresponding morphism $h\colon K\to L$ with $h\cdot [x_i] = [k(x_i)]$ then fulfils $Eh=k$.

(b) We denote by $\ck'$ the closure of $E[\ck]$ under isomorphisms in $\Sigma$-$\Met$.
It is equivalent to $\ck$. We prove that $\ck'$ is a variety by applying the Birkhoff Variety Theorem.

(b1) $\ck'$ is closed under products. Indeed, $U_\Sigma$ creates products (Remark \ref{R:pr}). Since $U_\Sigma \cdot E= \ck (G,-)$ preserves products, it follows that $E$ also preserves them. Given a collection of objects $K_i' \cong EK_i$ in $\ck'$, we have $\Pi K_i'\cong E(\Pi K_i) \in \ck'$.

(b2) Subalgebras. Fix an algebra $EK$, $K\in \ck$. To give a subalgebra means to give a subspace $m\colon M\hookrightarrow \ck (G,K)$ closed under the operations of $EK$. Our task is to find $L\in \ck$ such that $EL$ is isomorphic to the algebra $M$. The morphism $\widehat m\colon M\oti G\to K$ has a factorization $\widehat m=m'\cdot c$ as in Proposition \ref{P:fac}:
$$
\xymatrix@=2pc{
M \otimes G\ar[rr]^{\widehat m} \ar[dr]_{c} && K\\
& L \ar[ur]_{m'}&
}
$$

Since $E$ is an equivalence  and $m'$ is an isometric embedding, so is $Em'$. Thus $Em' [EL]$ is a subalgebra of $EK$. We prove for every morphism $g\colon G\to K$ that
$$
g\in M \quad \mbox{iff}\quad g\in Em' [EL]\,.
$$
That is, $M$ is the same algebra as $Em' [EL]$: both are subspaces of $EK$ closed under  the operations.

If $g\in M$, the morphism $g_0\colon I\to M$ representing it fulfils $g =\widehat{m\cdot g_0}$ (Example \ref{E:n}). Since $\widehat{m\cdot g_0} =\widehat m \cdot (g_0 \oti G)$, we obtain the commutative square below:
$$
\xymatrix@C=2pc@R=2pc{
G \ar[r]^g \ar[d]_{g_0 \oti G} & K\\
M\oti G \ar[r]_<<<<c & L \ar[u]_{m'}
}
$$
Thus for $h= c\cdot (g_0 \times G)$ in $EL$ we have $g= Em'(h)\in Em' [EL]$.

Conversely, if $g=m'\cdot h$ for some $h\colon G\to L$,
we prove $g\in M$. Since $G$ is a subregular projective, the morphism $h$ factorizes through the subregular epimorphism $c$, say $h= c\cdot h'$. Next, the morphism $h'\colon G\to M\oti G$ factorizes through the subregular epimorphism $i\oti G$ where $i\colon |M|\to M$ is carried by identity (Remark \ref{R:tt}). Say,  $h' = (i\oti G)\cdot h''$ for some $h'' \colon G\to |M|\oti G$. Thus  $h=c \cdot (i \otimes G) \cdot h''$:
$$
\xymatrix@C=2pc@R=2pc{
k\otimes G \ar[r]^{j\otimes G} & |M| \otimes G \ar[r]^{i\otimes G} & M\otimes G \ar[r]^{\widehat m} \ar[d]_c& K\\
G\ar[u]^{\sigma} \ar[ur]_{h''}\ar[rr]_h && L \ar[ur]_{m'} &
}
$$ 
Since $G$ is abstractly finite, there is a finite subobject $j\colon k\hookrightarrow |M|$ (where $k=\{0, \dots , k-1\}$) in $\Set$ such that $h''$ factorizes through $j\oti G$;  we call the factorizing morphism $\sigma \colon G\to k\oti G$.
 Thus $\sigma \in \Sigma_k$.
 
We obtain a $k$-tuple $i \cdot j$ in $M\subseteq EK$ given by $m\cdot i\cdot j\colon k\to \ck(G,K)$. The corresponding map from $\coprod\limits_{k} G$ to $K$ is $\widehat{m\cdot i\cdot j}$, therefore the operation $\sigma_{EK}$ yields
$$
\sigma_{EK} (m\cdot i\cdot j) = \widehat{m\cdot i\cdot j} \cdot \sigma\,.
$$
Now the composite of the upper row of the above diagram is
$$
\widehat m \cdot \big([i\cdot j] \oti G) = \widehat{m\cdot i\cdot j}\,.
$$
Thus the diagram yields
$$
m'\cdot h = \sigma_{EK} (m\cdot i \cdot j)\,.
$$
Since the subspace $m\colon M\hookrightarrow \ck(G,K)$ is closed under $\sigma_{EK}$, this proves that $g= m'\cdot h$ lies in $M$.

\vskip 1mm
(b3) Homomorphic images. Fix an algebra $EK$, $K\in \ck$. Given a surjective homomorphism    $e\colon EK\to A$ in $\Sigma$-$\Met$, we prove $A\in \ck'$. Let $D\colon \cb \to \Sigma$-$\Met$ be the kernel diagram (Example \ref{D:ker}) of $e$, then $e$ is its colimit map (Proposition \ref{P:sur} and Lemma \ref{L:pro}). Each $D\var$ (being a subalgebra of $EK\times EK$) lies in $\ck'$: use 
(b1) and (b2). Thus $D$ has a codomain restriction to $\ck'$. Since $E$ yields an equivalence functor $\ck \simeq \ck'$, there is a diagram $\bar D \colon \cb \to \ck$ with $D$ naturally isomorphic 
to $E\bar D$, say under
$\varphi \colon D\xrightarrow{\sim} E\bar D$, such that
$$
\bar Da = K \quad \mbox{and}\quad \varphi_a =\id_K\,.
$$

We verify that $\bar D$ is a subcongruence on $K$ (Definition \ref{D:proeq}).

(1) Each pair $\bar Dl_\var$, $\bar Dr_\var$ is $\var$-reflexive.
Indeed, let
  $q_1$, $q_2\colon Q\to K$ have distance $\var$, then so do $Eq_1$, $Eq_2\colon EQ\to Da$ by Item (a). Hence there exists $v\colon EQ \to D\var$ with $Eq_1 = Dl_\var \cdot v$ and $Eq_2= Dr_\var\cdot v$:
$$
\xymatrix@C=2pc@R=2pc{
& EQ \ar[dl]_{Eq_1} \ar[d]^v \ar[dr]^{Eq_2}&\\
EK \ar[d]_{\varphi_a=\id}& D\var\ar [l]^{Dl_\var}\ar[d]^{\varphi_\var} \ar[r]_{Dr_\var} & EK \ar[d]^{\varphi_a=\id}\\
EK & E\bar D\var \ar[r]_{E\bar Dr_\var} \ar[l]^{E\bar Dl_\var}& EK
}
$$
Since $E$ is full, there is  $u\colon Q \to \bar D\var$ with $\varphi _\var \cdot v=Eu$. Then the above diagram proves
$$
q_1 =\bar Dl_\var\cdot u \quad \mbox{and}\quad q_2 = \bar Dr_\var \cdot u
$$
since $E$ is faithful.

(2) Since $Dl_\var$, $Dr_\var$ is symmetric,  so is  $\bar Dl_\var$, $\bar Dr_\var$.

(3) $\ck'$ is closed under pullbacks in $\Sigma$-$\Str$ (due to (b1) and (b2)); thus, the transitivity condition for $D$ implies the transitivity condition for its codomain restriction. Hence $\bar D$ also satisfies this condition.

(4) Continuity of $D$ implies continuity of $\bar D$ since $E$ preserves limits.

We now apply the effectivity of subcongruences in $\ck$.
Let $\bar e \colon K\to \bar K$ be a morphism whose kernel diagram is the subcongruence $\bar D$ 
We can choose this morphism to be
a subregular epimorphism, using the factorization in Proposition 4.12.
From $G$ being a subregular projective it then follow that
$\bar  e \colon EK \to E\bar  K$ is also a subregular epimorphism, i.e
a surjective homomorphism. This homomorphism satisfies the compatibility
condition for $D$: use the natural isomorphism $\varphi$. Consequently,
$\bar  e$ factorizes through $e$: we have a homomorphism $h \colon A \to E\bar  K$ with
$$\bar e=h \cdot e.$$
To conclude the proof, we find an inverse morphism to $h$.

The inverse morphism $\bar  h \colon  E\bar  K  \to A$ takes $t \in E\bar  K$,
that is $t \colon G \to \bar  K$, factorizes it as $t=\bar e \cdot s$
for some $s \colon G \to K$ (using that $G$ is a subregular projective)
and applies $e$ to $s$:
$$ \hat  h (t)= e(s) \quad whenever\ t=\bar  e \cdot s.$$
Not only is the value $\bar  h (t)$ independent of the choice of $s$, $\bar  h$ is even
nonexpanding: given $t'=\bar  e \cdot s'$ in $E\bar  K$, we prove that
$$d(t,t') \geq d(e(s),e(s')).$$
Put $d(t,t')=\var$. Since $\bar  D$ is the kernel diagram of $\bar  e$,
from $(d(\bar  e \cdot s, \bar  e \cdot  s') \leq \var$ we conclude that there exists
$u \colon G \to \bar  D \var$ with
$$ s= \bar  D l_\var \cdot u \quad and \ s'= \bar  D r_\var \cdot u.$$
Morevover, the isomorphism $\varphi _\var \colon D \var \to E\bar  D \var$
yields $v \in D \var$ with $u= \varphi _\var (v)$. From $\bar  D l_\var \cdot
\varphi _\var = Dl_\var$ we then get
$$s=\bar  D l_\var \cdot \varphi _\var(v) = Dl_\var (v).$$
Analogously, $s'= Dr_\var (v)$. Applying the compatibility condition
$d(e \cdot Dl_\var , e \cdot Dr_\var) \leq \var$ to the element $v$
proves the desired inequality:
$$d(e(s),e(s'))=d(e \cdot Dl_\var (v), e \cdot Dr_\var (v)) \leq \var.$$

The nonexpanding map $\bar  h$ clearly fulfils $\bar  h \cdot \bar  e = e$.
Thus it is a homomorphism because $\bar  e$ is a surjective homomorphism
(Remark \ref{R:epi}). We have $h \cdot \bar  h = id$ due to
$$(h \cdot \bar  h) \cdot \bar  e = h \cdot e = \bar  e.$$
Analogously, $\bar  h \cdot h = id$ holds due to
$$(\bar h \cdot h) \cdot e = \bar  h \cdot \bar  e =e.$$

\end{proof}

%
%

\begin{remark} Analogously to Definition 2.6 a \emph{subeffective object} is
	an object whose hom-functor into $\Met$ preserves colimits of
	subcongruences.
\end{remark}

 It follows from Lemma \ref{L:pro} that in a category with $\var$-kernel pairs we have
\begin{center}
	subeffective \ $\Rightarrow$ \ subregular projective.
\end{center}

\begin{corollary}
	Varieties of quantitative algebras are, up to equivalence, precisely the
	metric-enriched categories with
	
	(1) Reflexive coequalizers;
	
	(2) An effective subvarietal generator.
	
\end{corollary}

Indeed, the necessity is clear because in Proposition \ref{P:ge} the object $G$ is effective since $U_\cv$ preserves colimits
of subcongruences. This follows from $\cv$ being closed under homomorphic images and $U_\Sigma$ preserving
colimits of subcongruences (Corollary \ref{C:cr}).

The proof of sufficiency is as in the preceding theorem, where the effectivity of subcongruences was only used
at the end for the definition of $\bar h$. It is easy to see that if $G$ is effective, then $E$ preserves colimits of
subcongruences, and we obtain $\bar h$ from the universal property of $E \bar e$.
\begin{example} \label{E:count} The category
$$
\Mon (\Met)
$$
of monoids in the monoidal category $(\Met, \otimes, 1)$ is not equivalent to a variety of quantitative algebras. Recall that these monoids have multiplication which is nonexpanding with respect to the addition metric $(3.1)$:
$$
d(xy, x'y') \leq d(x,x') + d(y, y')\,.
$$

Assuming that  we have a subvarietal generator $G$, we derive a contradiction by proving that $G$ fails to have tensors. We will namely prove that the left adjoint of the hom-functor of $G$ is not enriched. We  denote that hom-functor by
$$
H = \Mon (\Met) (G, -) \colon \Mon(\Met) \to \Met.
$$
The metric of the varietal generator  $G$ is discrete. Indeed, let $G_0$ be the underlying discrete monoid and $i\colon G_0\to G$ the identity-carried homomorphism.
It is a subregular epimorphism (Corollary \ref{C:var-epi}), thus a split epimorphism, since $G$ is a subregular projective. Consequently, $i$ is an isomorphism. Next consider the free discrete monoid $G_0 ^*$
on the set $G_0$,  and the canonical homomorphism $e\colon G_0 ^\ast \to G_0$. This, too, is a subregular (thus split) epimorphism. In the ordinary category of monoids the fact that we have a split epimorphism $e\colon G_0 ^\ast \to G_0$ implies that $G_0$ is a free monoid (\cite[Theorem 7.2.3]{AS}). Thus, we can assume that for some set $A$  our monoid $G$ is the monoid of words 
$$
G = A^\ast \quad \mbox{(discrete\  monoid)}.
$$
 Now observe first that the forgetful functor $U \colon \Mon(\Met)\to \Met$ has the left adjoint $F\colon \Met \to \Mon(\Met)$  taking a space $X$ to the monoid
$$
FX=\coprod_{n\in\N} X^n
$$
of finite words with the addition metric on $X^n$ (and, as usual, word concatenation as the multiplication). This functor $F$ is not enriched: consider the obvious pair
$$
f_a, f_b\colon \{0\} \to \{a,b\}\quad \mbox{where}\quad d(a,b)=1.
$$
It has distance $d(f_a, f_b)=1$. Then, using the addition metric, the maps $f_a^n, f_b^n$ have distance $n$, therefore $d(Ff_a, Ff_b)=\infty$.

The hom-functor $H$ of $G$ is naturally isomorphic to the functor $U^A$ assigning to a metric monoid $M$ the metric space $(UM)^A$. Its left adjoint $F_A\colon \Met \to \Mon(\Met)$ is given by the $A$-copowers of the free monoid: $F_AX = \coprod\limits_{A} FX$. Since $G$ is a generator, $A\ne \emptyset$. From $d(Ff_a, Ff_b)=\infty$ we conclude $d(F_A f_a, F_A f_b)=\infty$. Thus $F_A$ is not enriched, hence $G$ does not have tensors.
\end{example}

\end{document}